\documentclass[final]{siamart220329}
\usepackage[utf8]{inputenc}
\usepackage{listings}
\usepackage[framed]{matlab-prettifier}
\usepackage{amsmath,amsfonts,amssymb}
\usepackage{graphicx}
\usepackage{epstopdf}
\usepackage{algorithm,algorithmic}
\usepackage{caption}
\usepackage{subcaption}
\usepackage[percent]{overpic}

\DeclareMathOperator*{\argmin}{arg\,min}

\title{Are sketch-and-precondition least squares solvers numerically stable?}
\author{Maike Meier\thanks{Mathematical Institute, University of Oxford. (meier@maths.ox.ac.uk)} \and Yuji Nakatsukasa\thanks{Mathematical Institute, University of Oxford. (nakatsukasa@maths.ox.ac.uk)} \and Alex Townsend\thanks{Department of Mathematics, Cornell University. (townsend@cornell.edu)} \and Marcus Webb\thanks{Department of Mathematics, University of Manchester. (marcus.webb@manchester.ac.uk)}}
\date{\today}

\begin{document}

\maketitle

\begin{abstract}
Sketch-and-precondition techniques are efficient and popular for solving large least squares (LS) problems of the form $Ax=b$ with $A\in\mathbb{R}^{m\times n}$ and $m\gg n$. This is where $A$ is ``sketched" to a smaller matrix $SA$ with $S\in\mathbb{R}^{\lceil cn\rceil\times m}$ for some constant $c>1$ before an iterative LS solver computes the solution to $Ax=b$ with a right preconditioner $P$, where $P$ is constructed from $SA$. Prominent sketch-and-precondition LS solvers are Blendenpik and LSRN. We show that the sketch-and-precondition technique in its most commonly used form is not numerically stable for ill-conditioned LS problems. For provable and practical backward stability and optimal residuals, we suggest using an unpreconditioned iterative LS solver on $(AP)z=b$ with $x=Pz$. Provided the condition number of $A$ is smaller than the reciprocal of the unit round-off, we show that this modification ensures that the computed solution has a backward error comparable to the iterative LS solver applied to a well-conditioned matrix. Using smoothed analysis, we model floating-point rounding errors to argue that our modification is expected to compute a backward stable solution even for arbitrarily ill-conditioned LS problems. Additionally, we provide experimental evidence that using the sketch-and-solve solution as a starting vector in sketch-and-precondition algorithms (as suggested by Rokhlin and Tygert in 2008) should be highly preferred over the zero vector. The initialization often results in much more accurate solutions---albeit not always backward stable ones.

\end{abstract}



\section{Introduction}\label{sec:intro}
Randomized numerical linear algebra is a growing subfield of matrix computations that has produced major advances in low-rank approximation~\cite{martinsson2020randomized}, iterative methods~\cite{tropp2022randomized}, and projections~\cite{boutsidis2009random}. Sketch-and-precondition techniques are a class of randomized algorithms for solving overdetermined least squares (LS) problems of the form 
\begin{equation}\label{eq:LeastSquaresProblem} 
\min_{x\in\mathbb{R}^n}\|Ax-b\|_2,
 \qquad A\in\mathbb{R}^{m\times n}, \quad b\in\mathbb{R}^{m\times 1},
\end{equation}
where $m>n$. One first sketches $A$ to a smaller matrix $SA$ with a random sketch matrix
$S\in\mathbb{R}^{\lceil cn\rceil \times m}$ for some constant $c>1$, then a right preconditioner, $P$, is constructed from $SA$. Finally, one solves $Ax=b$ using an iterative LS solver with the right preconditioner $P$. There are many details of sketch-and-precondition based on how to sketch and construct $P$ as well as which iterative LS solver to employ. One of the most prominent sketch-and-precondition techniques is known as Blendenpik~\cite{Avron2010a} (see~\cref{alg:blendenpik}). In exact arithmetic, Blendenpik has a complexity of $\mathcal{O}(mn\log m)$ operations, which is better than the $\mathcal{O}(mn^2)$ QR-based direct solver.
Consequently, for large LS problems, Blendenpik can be substantially faster than the LS solver implemented in LAPACK~\cite{Avron2010a}.  However, are sketch-and-precondition techniques---such as Blendenpik---numerically stable?

\begin{algorithm}[h!]
		\caption{A sketch-and-precondition LS solver for~\cref{eq:LeastSquaresProblem} known as Blendenpik. Here, HHQR refers to the Householder QR algorithm.}\label{alg:blendenpik-standard}
	\begin{algorithmic} [1]
	\STATE Draw a random sketching matrix $S\in\mathbb{R}^{s\times m}$, where $m\gg s > n$
	\STATE Compute $B=SA$ 
	\STATE Compute the triangular factor $R$ of a QR factorization of $B$ using HHQR
	\STATE Solve $Ax=b$ with LSQR and right preconditioner $P = R^{-1}$
	\end{algorithmic}
 \label{alg:blendenpik}
\end{algorithm}

Surprisingly, we find that sketch-and-precondition techniques such as Blendenpik~\cite{Avron2010a} and LSRN~\cite{Meng2014} are numerically unstable in their standard form (see~\cref{fig:maincompintro}). 
For moderately ill-conditioned problems ($1\ll \kappa_2(A)\ll u^{-1}$, where $\kappa_2(A)=\sigma_{\max}(A)/\sigma_{\min}(A)$ is the condition number of $A$ and $u$ is the unit round-off), sketch-and-precondition iterations stagnate in terms of residual and backward error, potentially before optimal levels are reached. The main focus of this work is deriving a provable method to resolve these numerical instabilities. We suggest a modification to the sketch-and-precondition framework to obtain a new algorithm, which we coin sketch-and-apply. We show using classical stability analysis and experimentally that sketch-and-apply attains backward stable solutions under modest conditions. However, sketch-and-apply requires $\mathcal{O}(mn^2)$ operations, the same complexity as the classical least square solver based on QR. This work thus highlights the significant open problem: is there a fast (randomized) least-squares solver with guaranteed backward stability?

Throughout the paper, we assume that $A\in\mathbb{R}^{m\times n}$ and $SA\in\mathbb{R}^{s\times n}$ (for some $s$ such that $m>s>n$) are of full rank so that~\cref{eq:LeastSquaresProblem} has a unique solution.

\subsection{Sketch-and-precondition algorithms with a sketch-and-solve initialization} Sketch-and-precondition algorithms use sketching to construct a preconditioner to be used in an iterative LS method. Given an embedding matrix $S\in\mathbb{R}^{s\times m}$, $n<s\ll m$, and the resulting sketch $SA$, one would usually compute a QR decomposition of the sketch and use the inverse of the R-factor as a right preconditioner in LSQR. This is the basis of the Blendenpik algorithm~\cite{Avron2010a}. Another popular sketch-and-precondition technique is LSRN~\cite{Meng2014}, which uses the singular value decomposition (SVD) of $SA$ instead of the QR decomposition. The computational cost of Blendenpik is $\mathcal{O}(mn\log n)$ to compute the sketch, $\mathcal{O}(n^3)$ for the QR decomposition, and $\mathcal{O}(mn)$ cost per iteration in the iterative solver.

The most commonly used iterative solver for~\cref{eq:LeastSquaresProblem} is LSQR~\cite{paige1982lsqr}, which is competitive when $A$ is large, sparse, and well-conditioned. The number of iterations required to reach a desired accuracy typically depends on the condition number of $A$~\cite[Ch. 7.4]{bjorck1996numerical}. As a result, a preconditioner is necessary when $A$ is ill-conditioned to ensure a reasonable speed of convergence. With 
careful sketching, the condition number of $AP$ with $P=R^{-1}$ from~\cref{alg:blendenpik} ($SA=QR$) is small with high probability in exact arithmetic~\cite[Lem.~1]{rokhlin2008fast} (see~\cref{lem:exactcond}). In finite precision arithmetic, the condition number of the computed $AP$ remains modest in size, provided that $\kappa_2(A)u\ll 1$, where $\kappa_2(A)=\sigma_{\max}(A)/\sigma_{\min}(A)$ is the condition number of $A$ and $u$ is the unit round-off (see~\cref{thm:mainTheoremYhat}). However, as we will see, although $AP$ is well-conditioned, the application of $P$ can cause numerical errors. The crux of the numerical instability in sketch-and-precondition is the repeated application of $P=R^{-1}$, which is about as ill-conditioned as $A$. Indeed, it is recommended in the Blendenpik paper that LSQR is avoided when $\kappa_2(R) > 1/(5u)$~\cite{Avron2010a}. However, we find that for moderately ill-conditioned systems (such as $\kappa_2(A) \approx \sqrt{u^{-1}}$), numerical errors affect the accuracy of the solutions commensurately (see~\cref{sec:instability}).

Sketch-and-precondition in its standard form (see~\cref{alg:blendenpik}) has $x_0 = 0$ as the initial guess in LSQR. However, as suggested in Rokhlin and Tygert's paper~\cite{rokhlin2008fast}, a more natural guess is readily available; the solution to the sketched LS problem, i.e., 
\begin{equation}\label{eq:sketchedLSproblem}
x_0 = \argmin\limits_{x\in\mathbb{R}^{n}}\|SAx-Sb\|_2.
\end{equation}
This can be computed directly with the QR decomposition of $SA$; the resulting algorithm is displayed in pseudocode in~\cref{alg:SAP-initial guess}. Although originally proposed in~\cite{rokhlin2008fast} as part of the sketch-and-precondition framework, most implementations (e.g.~\cite{Avron2010a, Meng2014}) do not mention this choice of initial guess as part of their algorithms.\footnote{In the C implementation of Blendenpik~\cite{Avron2010a}, sketch-and-solve initialization is available as an option, but the user needs to append `improve\_start\_point' to the parameters.} 
\begin{algorithm}[h!]
		\caption{A sketch-and-precondition LS solver for~\cref{eq:LeastSquaresProblem} with a sketch-and-solve solution as initial guess. Here, HHQR refers to the Householder QR algorithm.}\label{alg:SAP-initial guess}
	\begin{algorithmic} [1]
	\STATE Draw a random sketching matrix $S\in\mathbb{R}^{s\times m}$, where $m\gg s > n$
	\STATE Compute $B=SA$ and $c = Sb$
	\STATE Compute both $Q$ and $R$ of the QR factorization of $B$ using HHQR
        \STATE Compute initial guess $x_0 = R^{-1}Q^Tc$
	\STATE Solve $Ax=b$ with LSQR and right preconditioner $P = R^{-1}$ and initial guess $x_0$
	\end{algorithmic}
 \label{alg:blendenpik_initialguess}
\end{algorithm}

The sketch-and-solve solution typically attains an accuracy within a small multiple of the optimal accuracy (see~\cref{sec:sketchsolveinit}). Although the attainable accuracy of sketch-and-precondition algorithms with a random initial guess will stagnate before a desired accuracy when dealing with ill-conditioned problems (see~\cref{fig:maincompintro}(b)), our experiments show~\cref{alg:SAP-initial guess} attains optimal residuals in most cases, and is significantly better than the standard, trivial initial guess $x_0=0$. 
Furthermore, as the QR decomposition (or SVD) of the sketch $SA$ is necessary for any sketch-and-precondition solver, this initial guess is obtained practically for free. We urge practitioners to adopt this as standard practice for these types of algorithms. It should be noted, however, that there are instances where the solution found by sketch-and-precondition with sketch-and-solve initialization does not attain backward stable solutions (see~\cref{fig:maincompintro}(a)). To further resolve the numerical instabilities, we introduce the \textit{sketch-and-apply} algorithm.

\subsection{Sketch-and-apply} The sketch-and-apply algorithm is a modification of the sketch-and-precondition algorithm for which we can ensure that the computed residual is close to optimal and the backward error is approximately machine precision regardless of the initial guess. We also prove backward stability. 

The modification of sketch-and-precondition is simple. Instead of using $P = R^{-1}$ as a preconditioner, we explicitly apply $P$ by computing $AP$ and then employ an unpreconditioned iterative LS solver on $(AP)z=b$ with $x=Pz$. We therefore call this a \textit{sketch-and-apply} technique. Of course, in exact arithmetic sketch-and-precondition and sketch-and-apply compute the same solution; however, for ill-conditioned LS problems in floating-point arithmetic, we find a significant difference. To our knowledge, this is the first algorithm for LS problems based on randomized sketching that is demonstrated to be backward stable (with a mild assumption to be made precise in~\cref{thm:mainTheoremLSQR}). 

By computing $AP$ explicitly, we remove all ill-conditioning from the iterative solver and instead use an unpreconditioned solver on a well-conditioned system. This results in accurate and backward stable solutions. In particular, we prove that if $\kappa_2(A)\ll u^{-1}$, our sketch-and-apply technique computes a backward stable solution provided that LSQR on a well-conditioned matrix computes a backward stable solution (see~\cref{thm:mainTheoremLSQR}).
\begin{algorithm}[h!]
	\caption{A sketch-and-apply LS solver for~\cref{eq:LeastSquaresProblem}. Here, HHQR refers to the Householder QR algorithm. One can include a sketch-and-solve initial guess by computing the QR decomposition $QR = SA$ and $z_0 = Q^TSb$, and setting $z_0$ as an initial guess in step 5.}
 \label{alg:ARBlendenpik}
 \begin{algorithmic}[1]
    \STATE	Draw a random sketching matrix $S\in\mathbb{R}^{s\times m}$, where $m\gg s > n$
	\STATE Compute $B=SA$
	\STATE Compute the triangular factor $R$ of a QR factorization of $B$ using HHQR
        \STATE Compute $Y = AR^{-1}$ with forward substitution 
	
	\STATE Solve $Yz=b$ with LSQR and no preconditioner 
	\STATE Compute $x = R^{-1}z$ with back substitution 
 \end{algorithmic}
\end{algorithm}

Unfortunately, while sketch-and-precondition techniques cost $\mathcal{O}(mn\log m)$ operations, sketch-and-apply costs $\mathcal{O}(mn^2)$ operations as $AP$ must be computed. Therefore, sketch-and-apply techniques have the same computational complexity as the classical QR-based LS solver. We note, however, that LSRN, another popular sketch-and-precondition algorithm, also uses $\mathcal{O}(mn^2)$ operations but is still competitive as the expensive computation is in matrix multiplication, which is highly parallelizable. The same applies to sketch-and-apply.

When $\kappa_2(A) \gtrsim u^{-1}$, our analysis does not guarantee that~\cref{alg:ARBlendenpik} leads to an accurate solution. However, in practice, we often get accurate final LS residuals. To explain why this is the case, we model floating-point rounding errors using smoothed analysis. That is, we consider ``smoothing" $A$ to $A + \sigma G/\sqrt{m}$, where the entries of $G$ are independent and identically distributed (i.i.d.)~standard Gaussian random variables and $\sigma$ is a scaling factor that we select as $\sigma = 10\|A\|_2u$. The idea is that $A + \sigma G/\sqrt{m}$ is significantly better conditioned than $A$ itself, assuming that $A$ is extremely ill-conditioned. In fact, for sufficiently small $\sigma$,  one can show that $\kappa_2(A + \sigma G/\sqrt{m}) \lesssim 1/\sigma$ with high probability~\cite{burgisser2010smoothed} (see~\cref{cor:smoothedKappa}). The additive perturbation of $\sigma G/\sqrt{m}$ to $A$ ensures that sketch-and-apply techniques can also compute solutions with good backward error, even when $A$ is extremely ill-conditioned (see~\cref{alg:smoothedARBlendenpikshort}). For extremely ill-conditioned LS problems, one could explicitly add an additive random perturbation or hope that floating point rounding errors deliver the same effect, as it often does in practice.

\begin{algorithm}[h!]
	\caption{A smoothed sketch-and-apply LS solver for~\cref{eq:LeastSquaresProblem} when $\kappa_2(A) \gtrsim u^{-1}$.}\label{alg:smoothedARBlendenpikshort}
 \begin{algorithmic}[1]
    \STATE Draw a random standard Gaussian matrix $G\in\mathbb{R}^{m\times n}$ with i.i.d.~entries
    \STATE Compute $\tilde{A} = A + \sigma G/\sqrt{m}$ for $\sigma = 10\|A\|_2u$
    \STATE Perform~\cref{alg:ARBlendenpik} on $\tilde{A}$
 \end{algorithmic}
\end{algorithm}

Throughout the paper, we assume that $A\in\mathbb{R}^{m\times n}$ and $SA\in\mathbb{R}^{s\times n}$, $m>s>n$, are of full rank, so that~\cref{eq:LeastSquaresProblem} has a unique solution.

\subsection{Sketch-and-solve initialization versus sketch-and-apply}
~\cref{fig:maincompintro}(b), as well as various other figures throughout this work, shows the enormous practical significance of using the sketch-and-solve solution as an initial guess. In the version of this paper first submitted to the journal, we discussed the instability of sketch-and-precondition and presented sketch-and-apply as the fix. A careful reviewer suggested that we initialize sketch-and-precondition with the sketch-and-solve solution. The improvement in stability and speed of convergence of the sketch-and-precondition when using the sketch-and-solve
 initialization (\cref{alg:blendenpik_initialguess}) should be one of the main take-aways of this work from a practical viewpoint. However, it should be noted that sketch-and-precondition with sketch-and-solve initialization is not completely freed from the numerical instabilities present in standard sketch-and-precondition (see \cref{fig:maincompintro}(a)). As a result, sketch-and-apply has significance when it is important to retrieve backward stable solutions.

The remainder of this work focuses on analyzing the sketch-and-apply algorithm. As far as the authors are aware, this is one of the first rigorous stability analyses for a randomized algorithm.  As a result, it is also one of the first randomized algorithms proven to be backward stable.

\subsection{The numerical instabilities of sketch-and-precondition}\label{sec:instability}
We now demonstrate the numerical instabilities that occur with sketch-and-precondition techniques.\footnote{Experiments are performed in MATLAB 2022b using 64-bit arithmetic on a
single core of a MacBook Pro equipped with a 2.3 GHz Dual-Core Intel Core i5 processor and 16 GB 2133 MHz LPDDR3 of system memory.} We construct LS problems at random by setting $A=U\Sigma V^T$, where $U\in\mathbb{R}^{m\times n}$ and $V\in\mathbb{R}^{n\times n}$ are random orthogonal matrices from the Haar distribution
and $\Sigma=\mbox{diag}(\sigma_1,\ldots,\sigma_n)$, where $\sigma_1,\ldots,\sigma_n$ are logarithmically spaced between $1$ and $\sigma_n = 10^{-10}$. We set $m = 10000$ and $n=100$. The right-hand side of the LS problem is generated as $b=Ax^*+e$, where $x^*$ is a random vector with i.i.d. standard Gaussian entries and $e$ represents a noise vector. The noise is chosen to be orthogonal to the column space of $A$ so that $x^*$ is the exact solution to $Ax=b$. For this reason, the computed residual is bounded from below by $\|e\|_2 = \|Ax^*-b\|_2$. We consider two values of $\|e\|_2$: (a) $\|e\|_2=10^{-2}$ (see~\cref{fig:maincompintro}(a)) and (b) $\|e\|_2=10^{-12}$ (see~\cref{fig:maincompintro}(b)).

To solve the LS problems, we use (1) standard sketch-and-precondition Blendenpik (SAP, see~\cref{alg:blendenpik}), (2) sketch-and-precondition Blendenpik with a sketch-and-solve initial guess (SAP-SAS, see~\cref{alg:blendenpik_initialguess}), (3) standard sketch-and-apply Blendenpik (SAA, see~\cref{alg:ARBlendenpik}), and (4) sketch-and-apply Blendenpik with a sketch-and-solve initial guess (SAA-SAS, see the caption of~\cref{alg:ARBlendenpik}). The only difference between the sketch-and-precondition and sketch-and-apply algorithms is how they employ the preconditioner; the actual preconditioner is identical. We use a sketching matrix $S\in\mathbb{R}^{4n\times m}$, known as a subsampled randomized cosine transform (see~\cref{sec:SRTTs}). The tolerance is set to $10^{-14}$ and the maximum number of iterations to $50$.

\begin{figure}[htpb]
  \centering
    \begin{overpic}[width = \linewidth]{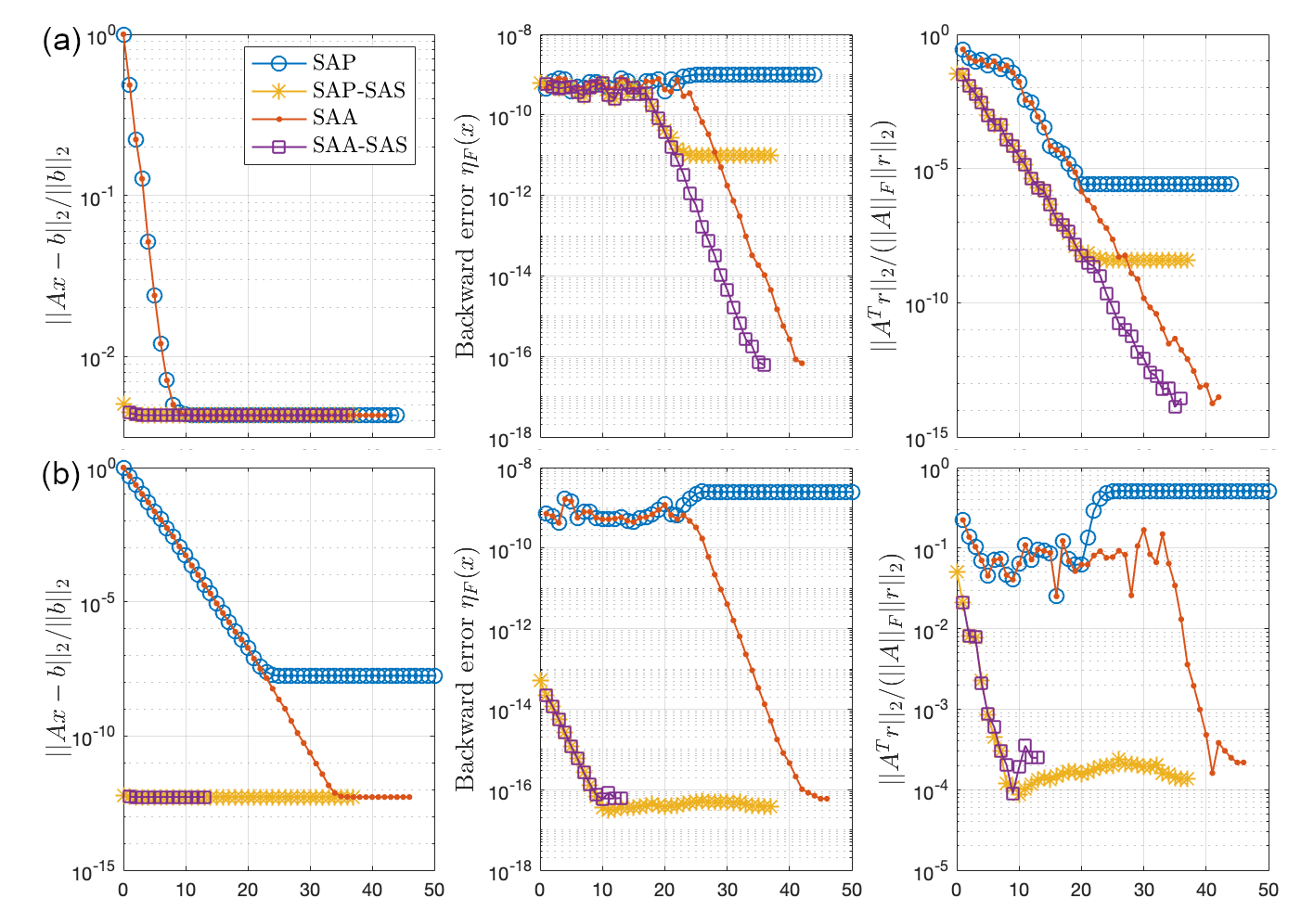}
        \put (20.5,55.5) {\tiny $\|e\|_2 = 10^{-2}$}
        \put (20,32) {\tiny $\|e\|_2 = 10^{-12}$}
    \end{overpic}
    \caption{Sketch-and-precondition Blendenpik (SAP) and sketch-and-apply Blendenpik (SAA), with and without sketch-and-solve initialization (SAS), applied to two noisy and randomly generated LS problems. The noise level is $\|Ax^*-b\|_2 = \|e\|_2 = 10^{-2}$ in the top row (a) and $\|Ax^*-b\|_2 = \|e\|_2 = 10^{-12}$ in the bottom row (b). The matrix $A$ is $10000\times 100$ with condition number $\kappa_2(A) = 10^{10}$. We display the residual, the backward error, and the relative normal residual for each LSQR iteration.}
    \label{fig:maincompintro}
\end{figure}

\cref{fig:maincompintro} compares the performance of the algorithms in terms of the relative residual $\|Ax-b\|_2/\|b\|_2$ (left column), the normwise backward error $\eta_F(x)$ (middle column, see~\cref{eq:defbackwarderrorLS}), and the relative normal residual $\|A^T(Ax-b)\|_2/(\|A\|_F\|Ax-b\|_2)$. A stable algorithm should compute a final solution with a relative residual close to the noise level given by $\|e\|_2$ and a backward error of order $u$, where the backward error is given by~\cite{walden1995optimal}
\begin{equation}\label{eq:defbackwarderrorLS}
    \eta_F(x) := \min\{\|[\Delta A, \Delta b]\|_F\, :\, \|(A+\Delta A)x - (b+\Delta b)\|_2 = \min\}.
\end{equation}
We compute it using~\cite{walden1995optimal} (see also~\cite[Thm.~20.5]{Higham2002})
$$ \eta_F(x) = \min\left\{\phi, \,\,\sigma_{\min}\left(\left[A\,\,\,\, \phi(I_m-rr^\dagger) \right]\right) \right\}, \quad \phi = \frac{\|r\|_2}{\sqrt{1 + \|x\|_2^2}}, \quad r = b-Ax,$$
where $\sigma_{\min}(B)$ indicates the smallest singular value of $B$.

Consider first~\cref{fig:maincompintro}(a): an inconsistent, moderately ill-conditioned problem. We see that all algorithms attain the optimal residual. However, considering the backward error, both sketch-and-precondition varieties fail to converge to a backward stable solution. The same behavior can be observed for the relative normal residual. In general we find that for inconsistent problems, i.e., problems with large optimal residual, sketch-and-precondition with or without initialization converges to a solution with good residual. However, for moderately ill-conditioned problems, these solutions are often not backward stable.

~\cref{fig:maincompintro}(b) shows the more obvious instabilities of sketch-and-precondition without initialization: the maximal attainable residual is not optimal. Consider the final residuals in~\cref{fig:maincompintro}(b). The sketch-and-precondition solution stagnates around $10^{-8}$, whereas all other algorithms attain $10^{-12}$. This instability is also visible in the backward error and relative residual, where we sketch-and-precondition without initialization is not backward stable.

Perhaps most notable about~\cref{fig:maincompintro}(b) is the success of sketch-and-solve initialization. The initial guess has an accuracy of the same order as the optimal solution in terms of residual (reflecting standard theory \cite{martinsson2020randomized}), and the following iterates usually quickly converge to the optimal residual. It appears that starting close to a good guess resolves much of the numerical instabilities of sketch-and-precondition Blendenpik. The initialization results in a backward stable solution. However, it must be noted that some floating point errors persist, as the backward errors for the inconsistent problems are not optimal (see ~\cref{fig:maincompintro}(a)).

Sketch-and-apply Blendenpik, with or without initialization, attains accurate solutions with a backward error of order $u$ in all cases. This supports our theoretical findings that sketch-and-apply is a backward stable algorithm (see~\cref{sec:condition}). Again, LSQR converges faster when initialization is used, and we always recommend doing this.

In all of the experiments, the tolerance tol in LSQR is set to be very small to allow us to investigate the maximal attainable accuracy. This checks both the relative normal residual $\|A^Tr\|_2/(\|A\|_F\|r\|_2)\leq {\rm tol}$ and the residual $\|Ax-b\|_2/\|b\|_2\leq {\rm tol}$. 
For highly inconsistent problems, our small choice of tol can result in many iterations before LSQR is stopped. It should be noted however, that the initialized algorithms reach the levels they will stagnate on much sooner. The LSQR tolerance is thus a delicate aspect, and significant research has been devoted to the subject~\cite{chang2009stopping,hallman2020estimating,stewartLS}.
\begin{figure}
    \centering
    \includegraphics[width = \linewidth]{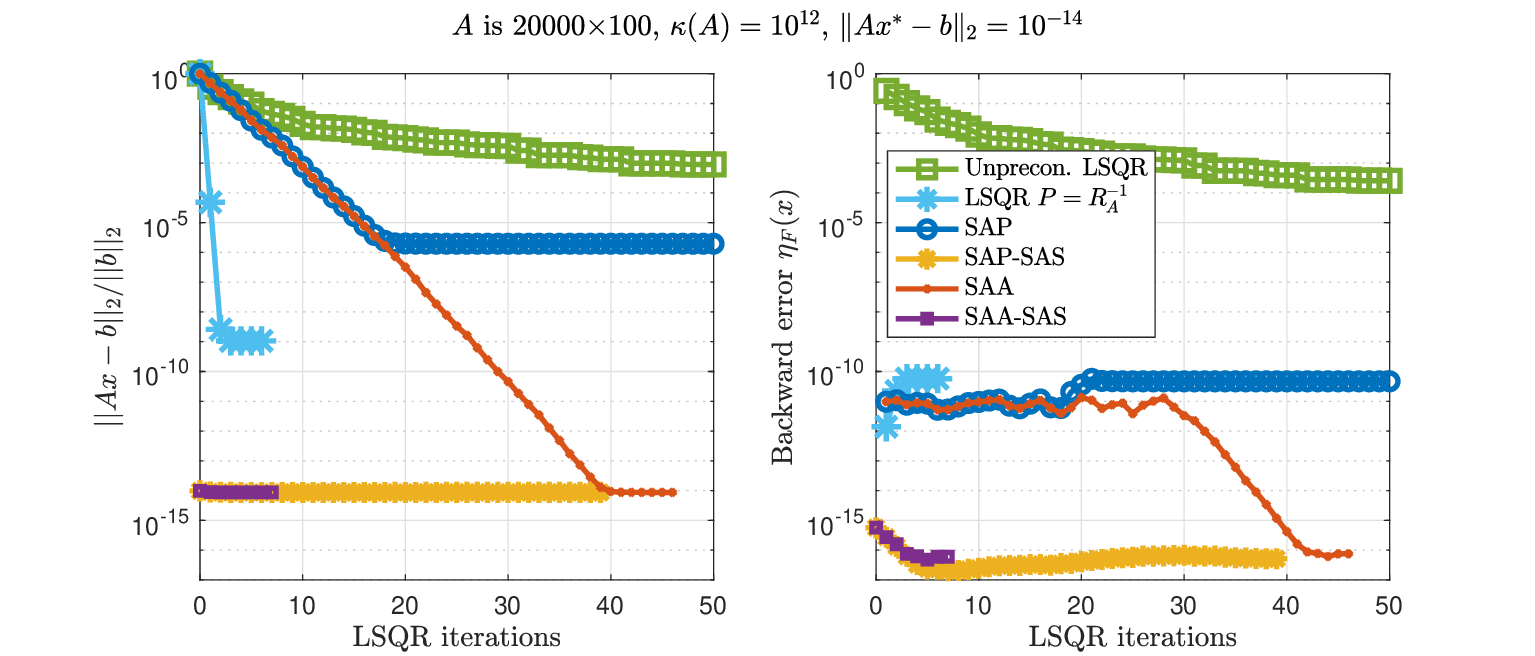}
    \caption{A comparison of the convergence of the relative residual and backward errors for unpreconditioned LSQR, preconditioned LSQR with $P = R_A^{-1}$ from the QR factorization of $A = Q_AR_A$, sketch-and-precondition (SAP) Blendenpik with and without sketch-and-solve (SAS) initialization, sketch-and-apply (SAA) Blendenpik with and without sketch-and-solve (SAS) initialization. The tolerance is set to machine precision.}
    \label{fig:LSQRcompintro}
\end{figure}

The numerical instabilities in standard sketch-and-precondition Blendenpik could be due to multiple sources. 
To demonstrate that it is the way that the preconditioner is employed, and not how it is formed, we try one more randomly generated LS problem with $A\in\mathbb{R}^{20000\times 100}$, $\kappa_2(A) = 10^{12}$, and a noise level of $\|e\|_2=10^{-14}$. We solve the LS problem with unpreconditioned LSQR and preconditioned LSQR, where the preconditioner is obtained by computing the QR factorization of $A$. This approach is not practical because it requires the QR factorization of $A$, but is done to demonstrate the quality of the preconditioners is not at fault. We find in \cref{fig:LSQRcompintro} that the sketch-and-solve initialization approach computes accurate backward-stable solutions when combined with either sketch-and-precondition or sketch-and-apply. For the various algorithms employing $x_0 = 0$ as an initial guess, sketch-and-apply Blendenpik is the only approach that computes a solution with a residual close to optimal and a backward error close to machine precision. We also note that the qualitative behavior shown in~\cref{fig:maincompintro,fig:LSQRcompintro} is unaffected by different types of sketch matrices or sketch dimensions. These choices influence $\kappa_2(AP)$ and $\kappa_2(P)$ slightly, which can somewhat affect the convergence behavior in terms of speed and maximal attainable accuracy. However, we do not observe qualitative differences provided that $\kappa_2(AP) = \mathcal{O}(1)$.

While we have not spotted the precise source of instability in the standard preconditioned LSQR routines, we believe it lies in the fact that each application of the preconditioner $P$---which involves solving an ill-conditioned linear system $Rx=b$---incurs a relative error proportional to $u\kappa_2(R)$. Such errors are present every time we apply $P$ or $P^T$. 
Moreover, each application behaves somewhat differently, in that in the $i$th iterate we have $(R+\Delta R_i)\hat x_i=b_i+\delta b_i$ where 
$\Delta R_i,\delta b_i$ are small but different for each $i$. 
In other words, one can view the preconditioner as having an $u\kappa_2(R)$ nonlinear effect. The fact that the convergence of iterative methods can get impaired by nonlinear preconditioners has been observed in~\cite{WathenReesetna08}. 

\subsection{Paper structure}
In~\cref{sec:preliminaries}, we introduce some background material on sketch-and-precondition LS solvers, sketching, numerical stability analysis, and smoothed analysis. In~\cref{sec:condition}, we consider sketch-and-apply Blendenpik (see~\cref{alg:ARBlendenpik}) in finite precision. In~\cref{sec:SmoothedAnalysisMain}, we consider extremely ill-conditioned LS problems and look at the numerical stability of smoothed sketch-and-apply Blendenpik (see~\cref{alg:smoothedARBlendenpikshort}). We introduce a master algorithm and display numerical experiments in~\cref{sec:masteralgorithmandnumexp}. Finally, in~\cref{sec:discussion}, we conclude by noting the practical significance of the sketch-and-solve initialization, and discuss the wider implications of the observed instabilities of standard sketch-and-precondition.

While in this paper we focus on LS problems, we expect much of the stability results to carry over to solving underdetermined linear systems using sketching, as done in LSRN~\cite{Meng2014}.

\section{Background material}\label{sec:preliminaries}
We now introduce some background material for sketch-and-precondition techniques (see~\cref{sec:SketchAndPrecondition}), random sketching matrices (see~\cref{sec:Sketching}), numerical stability analysis (see~\cref{sec:prelimsStability}), and smoothed analysis (see~\cref{sec:SmoothedAnalysis}). 

\subsection{Sketch-and-precondition least squares solvers}\label{sec:SketchAndPrecondition}
The idea behind the sketch-and-precondition technique is that a reasonable preconditioner for the LS problem in~\cref{eq:LeastSquaresProblem} can be constructed from a sketch of $A$. In Blendenpik, the matrix $A$ is sketched to a small tall-skinny matrix $SA$ and the preconditioner is taken to be the inverse of the upper triangular factor from a QR factorization of $SA$, i.e., $P=R^{-1}$. In exact arithmetic, the condition number of $AP$ is equal to $SQ_A$, where $Q_A$ is an orthonormal basis for the column space of $A$. 

\begin{lemma}\label{lem:exactcond}  Let $A \in \mathbb{R}^{m\times n}$ have linearly independent columns, $S\in\mathbb{R}^{s\times m}$ with $s\geq n$ have linearly independent rows, and $B = SA$. If $A = Q_A R_A$ and $B = Q R$ are economized QR factorizations of $A$ and $B$, respectively, then 
\[
\kappa_2(AP) = \kappa_2(SQ_A).
\]
where $P = R^{-1}$. 
\end{lemma}
\begin{proof}
The proof follows the same argument by Meng, Saunders, and Mahoney in~\cite[Lemma 4.2]{Meng2014} and Rokhlin and Tygert in~\cite[Thm. 1]{rokhlin2008fast} Note $AP=Q_AR_AR^{-1}$ so that
$\kappa_2(AP)=\kappa_2(R_AR^{-1})$. Now, since $B=SA=QR=S(Q_AR_A)$, we have that $(SQ_A)R_AR^{-1}=Q$ is orthonormal. Hence, $\kappa_2(R_AR^{-1})=\kappa_2(SQ_A)$. 
\end{proof}

\Cref{lem:exactcond} shows that the main idea behind sketch-and-precondition techniques is excellent in exact arithmetic. In particular, the value of $\kappa_2(AP)$ is independent of $\kappa_2(A)$. In particular, one expects rapid convergence of iterative LS solvers using the preconditioner $P$. While this is correct in exact arithmetic, we have seen that rounding errors cause significant problems for ill-conditioned LS problems (see~\cref{sec:instability}). Regardless, in exact arithmetic, we are left with the task of designing a sketching matrix so that $\kappa_2(SQ_A)$ is close to $1$ with high probability. 

\subsubsection{Sketch-and-solve initialization}\label{sec:sketchsolveinit}
Rokhlin and Tygert~\cite{rokhlin2008fast} included a sketch-and-solve initial guess in the description of their original sketch-and-precondition algorithm. This initial guess is the solution to $\min\|S(Ax-b)\|_2$ computed with a direct solver. It can be computed using the QR decomposition of $SA$, which is necessary in any case. The accuracy of sketch-and-solve solutions are well-studied and bounds are generally of the form~\cite{martinsson2020randomized}
$$\|Ax_{0}-b\|_2\leq \frac{1+\epsilon}{1-\epsilon}\|Ax^*-b\|_2 \quad \text{ for } s\sim n\log(n)/\epsilon^2,$$
where $x_0$ is the sketch-and-solve solution (see~\eqref{eq:sketchedLSproblem}), $x^*$ is the optimal solution to~\eqref{eq:LeastSquaresProblem}, and $s$ is the sketch dimension ($S\in\mathbb{R}^{s\times m})$. Note that $\epsilon$ is the subspace embedding constant of $[A,b]$, the concatenation of $A$ and $b$, and is usually modestly small, say $\epsilon = 0.5$. In the language used in Lemma \ref{lem:exactcond}, $\frac{1+\epsilon}{1-\epsilon} = \kappa_2(SQ_{[A,b]})$, whose value can be bounded using identical arguments to those for $\kappa_2(SQ_A)$, which will be discussed below. It follows that for a sufficiently large sketch dimension, the accuracy (in terms of residual) of the sketch-and-solve solution is of the same order as that of the optimal solution.

\subsection{Sketching matrices}\label{sec:Sketching}
Given~\cref{lem:exactcond}, it is paramount to understand how to construct a sketching matrix $S$ so that $SQ_A$ is well-conditioned.  Matrices that achieve this are called sketching matrices. We generally desire $m\gg s > n$ so that $SA$ has full column rank and computing a preconditioner from $SA$ is computationally efficient. 

There are various ways to construct sketching matrices; almost all are randomly generated. We consider two important types here: (1) Gaussian matrices~\cite{Meng2014, vershynin2010introduction, woodruff2014sketching}, and (2) subsampled randomized trigonometric transforms (SRTTs)~\cite{Avron2010a,martinsson2020randomized, rokhlin2008fast, tropp2011improved}. Other sketching techniques include random sampling~\cite{ipsen2014effect}, sparse embeddings~\cite{clarkson2017low}, and hashing matrices~\cite{cartis2021hashing}.

\subsubsection{Gaussian sketching matrices}
A Gaussian sketching matrix of size $s\times m$ is a matrix with i.i.d.~Gaussian entries of mean $0$ and variance $1/s$. If one selects $s \geq \lceil cn\rceil$ with $c>1$, then one can show that $\kappa_2(SQ_A)$ is a small constant depending on $c$ with high probability~\cite{martinsson2020randomized}. The drawback of Gaussian matrices is the cost of computing $SA$. In particular, $SA$ costs $\mathcal{O}(mn^2)$ operations as $s = \mathcal{O}(n)$. However, Gaussian sketching matrices are often used in theoretical analysis due to the availability of excellent probability theory and concentration of measure arguments~\cite{halko2011finding, martinsson2020randomized}. 

\subsubsection{Subsampled randomized trigonometric transforms}\label{sec:SRTTs}
An SRTT has the form $S = \mathcal{S}FD$, where $D$ is a square diagonal matrix with $\pm 1$ entries at random, $F$ is an orthogonal trigonometric transform (e.g., Fourier, cosine, or Hadamard), and $\mathcal{S}$ is an $s\times m$ scaled sampling matrix with one non-zero entry per row. For an SRTT, if $s \geq \lceil c n\log n\rceil$ with some constant $c>1$, then $SQ_A$ is well-conditioned with high probability~\cite{tropp2011improved}. In practice, the $\log n$ factor in $s$ can often be dropped~\cite{martinsson2020randomized}. Moreover, the matrix-matrix product $SA$ can be computed in $\mathcal{O}(mn\log s)$ operations, although the usual computational cost is $\mathcal{O}(mn\log m)$ associated with computing $\mathcal{S}(F(DA))$, as fast implementations for the SRTT are not readily available. We use subsampled randomized cosine transforms as the sketching matrices in our numerical experiments.

\subsection{Numerical stability analysis}\label{sec:prelimsStability}
Basic arithmetic operations (e.g., $+$, $-$, $\cdot$, and $/$) on a computer are performed with rounding errors due to the finite precision of numbers in floating-point representation. The stability of a numerical algorithm refers to the property to compute solutions with a small backward error~\cite{Higham2002} in the presence of rounding errors. We are particularly interested in proving that sketch-and-apply Blendenpik computes backward stable solutions. The backward error for LS problems is defined in~\cref{eq:defbackwarderrorLS}. A backward stable algorithm computes solutions with a backward error of the same order as the precision.

We assume that computations are performed with a precision of $u\ll 1$ and that numbers and arithmetic operations are exact up to this precision. We now consider three example algorithms: (1) The standard algorithm for matrix-matrix multiplication, (2) Householder QR, and (3) Triangular solve with substitution. We need these results to understand sketch-and-apply LS solvers. 

\subsubsection{Matrix-matrix multiplication} 
Consider matrix-matrix multiplication of two matrices $A\in\mathbb{R}^{m\times n}$ and $B\in\mathbb{R}^{n\times p}$. It is shown that the standard algorithm for computing $AB$ satisfies~\cite[Chapt.~3]{Higham2002}
\begin{equation}
    |AB - {\rm fl}(AB)| \leq \frac{nu}{1-nu}|A||B| = \gamma_n |A||B|,\qquad \gamma_n = \frac{nu}{1-nu},
    \label{eq:matmatmult}
\end{equation}
where ${\rm fl}(AB)$ means that $AB$ is computed with precision $u$ and $|A|$ denotes the entry-wise absolute value of $A$. The inequality in~\cref{eq:matmatmult} is understood as holding for each entry. It can be shown that~\cref{eq:matmatmult} leads to the bound 
\begin{equation}\label{eq:matmatmulterr1}
    \|AB - {\rm fl}(AB)\|_2 \leq \gamma_n\min(\sqrt{m},\sqrt{n})\min(\sqrt{n}, \sqrt{p})\|A\|_2\|B\|_2,
\end{equation}
where $\|\cdot\|_2$ is the spectral norm. The inequality in~\cref{eq:matmatmulterr1} is a forward error in the sense that it shows that ${\rm fl}(AB)$ is close to $AB$. 

\subsubsection{Householder QR factorization} 
Next, we consider the Householder QR of a matrix $A\in\mathbb{R}^{m\times n}$, where $m\geq n$. We denote the triangular factor computed by the Householder QR algorithm as $\hat{R}\in\mathbb{R}^{n\times n}$, and we are interested in the accuracy of $\hat{R}$. By~\cite[Thm.~19.4]{Higham2002}, there exists a matrix $Q\in\mathbb{R}^{m\times n}$ with orthonormal columns such that 
\begin{equation*}
    A + E = Q\hat{R},\qquad \|e_j\|_2\leq \tilde{\gamma}_{mn}\|a_j\|_2, \quad 1\leq j\leq n,
\end{equation*}
where $e_j$ and $a_j$ denote the $j$th columns of $E$ and $A$ respectively, and $\tilde{\gamma}_{mn} = cmnu/(1-cmnu)$ for a small integer constant $c$. A direct consequence is that 
\begin{equation}\label{eq:QRerr}
    \|E\|_F\leq \tilde{\gamma}_{mn}\|A\|_F,
\end{equation}
where $\|\cdot\|_F$ is the matrix Frobenius norm.  The inequality in~\cref{eq:QRerr} has the interpretation that the QR factorization computed by Householder QR is exact for a slightly perturbed matrix $A$. 

\subsubsection{Triangular system solving with substitution} 
Finally, we consider the error arising from solving a triangular system with substitution. Let $R\in\mathbb{R}^{n\times n}$ be a nonsingular upper-triangular matrix. It is known that the computed solution, $\hat{x}$, of the linear system $R^Tx=b$ satisfies~\cite[Sec.~8.2]{Higham2002}
\begin{equation}
    \|x - \hat{x}\|_2\leq \frac{\sqrt{n}\gamma_n\kappa_2(R)}{1-\sqrt{n}\gamma_n\kappa_2(R)}\|x\|_2,
    \label{eq:TriangularError} 
\end{equation}
where $\gamma_n$ is defined in~\cref{eq:matmatmult}. 
The inequality in~\cref{eq:TriangularError} tells us that we expect the relative forward error of the computed solution to be on the order of $\kappa_2(R)$. 

When we do any sketch-and-apply technique, we need to compute $AP$, where $P=R^{-1}$. We compute $AP$ by solving $R^{T}x_i=a_i$ for $1\leq i\leq m$, where $a_i^T$ denotes the $i$th row of $A$ and $x_i^T$ is the $i$th row of $AP$. If we denote the computed solution to $R^{T}x_i=a_i$ by $\hat{x}_i$, then from~\cref{eq:TriangularError} we find that
\begin{align*}
    \sum_{i=1}^m \|x_i-\hat{x}_i\|_2^2 \leq \left[\frac{\sqrt{n}\gamma_n\kappa_2(R)}{1-\sqrt{n}\gamma_n\kappa_2(R)} \right]^2\sum_{i=1}^m\|x_i\|_2^2.
\end{align*}
We conclude that for $P = R^{-1}$ we have
\begin{equation}\label{eq:trisubserr}
    \|AP - \widehat{AP}\|_F \leq \frac{\sqrt{n}\gamma_n\kappa_2(P)}{1-\sqrt{n}\gamma_n\kappa_2(P)}\|AP\|_F,
\end{equation}
where $\widehat{AP}$ denotes the computed matrix-matrix product. Roughly speaking, $\widehat{AP}$ is close to $AP$ when $\kappa_2(P)$ is modest. We additionally bound the backward error resulting from solving $Rx = b$ with back substitution. The computed solution $\hat{x}$ satisfies \cite[Theorem 8.5]{Higham2002}
\begin{equation}\label{eq:backerrbacksub}
(\hat{R} + \Delta \hat{R})\hat{x} = b, \quad \|\Delta \hat{R}\|_2 \leq \gamma_n\sqrt{n} \|\hat{R}\|_2.
\end{equation}

\subsection{Smoothed analysis of condition numbers}\label{sec:SmoothedAnalysis}
Due to rounding errors, most matrices with $\kappa_2(A)>u^{-1}$ are perturbed to a matrix with condition number less than $u^{-1}$ once represented on a computer. 
To explain this, we can model the rounding errors by an additive Gaussian perturbation to $A$. This type of technique fits into the field of smoothed analysis and has been studied in \cite{higham2019new}. More precisely, we suppose that $A\in\mathbb{R}^{m\times n}$ and that we would like to bound the condition number of $A + \sigma G$, where $\sigma$ is a small number and $G$ is a  standard Gaussian matrix with i.i.d.~entries. We have the following statement, which is a corollary of~\cite[Thm.~1.1]{burgisser2010smoothed}.
\begin{corollary}
Let $A\in\mathbb{R}^{m\times n}$ with $m\geq 3n$ and $r\geq 10^{4}$. If $A$ is perturbed to $\hat{A} = A  + \sigma G/\sqrt{m}$ with $\sigma = 8.25\|A\|_2/r$, where $G$ is standard Gaussian matrix with i.i.d.~entries, then
\[
    \mathbb{P}\!\left[\kappa_2(\hat{A}) \geq r\,\right] < 2^{n-m}.
\]
\label{cor:smoothedKappa}
\end{corollary}
\begin{proof} The result immediately follows by setting $z = (r(1-\lambda))/e$ and $\sigma = 8.25/(r\sqrt{m})$ in Theorem 1.1 of~\cite{burgisser2010smoothed}. This takes into account that Theorem 1.1 of~\cite{burgisser2010smoothed} considers the case when $m<n$. The extra $\|A\|_2$ factor of $\sigma$ in the statement of the corollary ensures that our final result continues to hold when $\|A\|_2\neq 1$.
\end{proof} 
\Cref{cor:smoothedKappa} shows us that a small additive Gaussian perturbation to a matrix $A$ ensures that the condition number is small with high probability. It partially explains why matrices represented in floating-point arithmetic rarely have a condition number $>u^{-1}$. In~\cref{sec:SmoothedAnalysisMain}, we use~\cref{cor:smoothedKappa} to explain why sketch-and-apply techniques continue to deliver accurate LS solutions, even when we have $\kappa_2(A)>u^{-1}$. 

\section{On the numerical stability of sketch-and-apply Blendenpik}\label{sec:condition}
By the properties of sketching matrices and~\cref{lem:exactcond}, we know that if one computes $P$ and $AP$ in exact arithmetic, then $\kappa_2(AP)$ is bounded by a constant independent of $\kappa_2(A)$ with high probability. In finite precision, one typically needs to be more careful. In this section, we show that when $\kappa_2(A)\ll u^{-1}$, the condition number of the computed matrix $AP$ can also be bounded by a constant with high probability (see~\cref{thm:mainTheoremYhat}). Unfortunately, such a bound does not extend to the case when $\kappa_2(A)> u^{-1}$ (see~\cref{sec:SmoothedAnalysisMain}). 

There are three main steps to compute $AP$. In finite precision, the computed matrices satisfy the following:
\begin{enumerate}
    \item $\hat{B}=SA + E_1$,
    \item $\tilde{Q}\hat{R} = \hat{B} + E_2$,
    \item $\hat{Y} = A\hat{R}^{-1} + \Delta Y$,
\end{enumerate}
where hats denote that we are accounting for rounding errors.
The matrix-matrix multiplication error $E_1$ can be bounded using~\cref{eq:matmatmulterr1}. We have a bound on the error arising from the Householder QR algorithm $E_2$ in~\cref{eq:QRerr}. Note that $\tilde{Q}$ has orthonormal columns but is not equal to $Q$, where $Q$ is the orthonormal factor of the exact QR factorization of $B = SA$. Finally, the error term $\Delta Y$ in the step resulting from solving the triangular matrix system can be bounded using~\cref{eq:trisubserr}. After all these errors, we want to derive a bound on $\kappa_2(\hat{Y})$, as $\hat{Y}$ is the computed preconditioned matrix in finite precision.

\subsection{Bounding $\mathbf{\boldsymbol\kappa_2(\hat{Y})}$ in terms of numerical errors}
First, let $E = E_1 + E_2$ be the sum of the additive errors from the matrix-matrix multiplication and the QR factorization, such that $\tilde{Q}\hat{R}  = QR + E$. We show in~\cref{lem:condsRandARinv} that the quantity $\varepsilon_1 := \|E\|_2\|R^{-1}\|_2$ controls the numerical error in the first two steps.

\begin{lemma}\label{lem:condsRandARinv}
    Let $A = Q_AR_A$ and $SA = QR$ be economized QR factorizations of $A$ and $B = SA$, respectively. Let $\hat{B}$ be the computed $SA$ and $\hat{R}$ be the computed upper triangular factor from the Householder QR algorithm of $\hat{B}$, both computed in finite precision. Let $\tilde{Q}$ be a matrix with orthonormal columns such that $E = \tilde{Q}\hat{R} - QR$ is the numerical error associated with $\hat{R}$. Define $\varepsilon_1 := \| E \|_2\|R^{-1}\|_2$. Assuming that $\varepsilon_1<1$, we find that $\hat{R}$ is invertible and $A\hat{R}^{-1}$ is full rank. Furthermore,
    $$\kappa_2(\hat{R})\leq \frac{\kappa_2(R) + \varepsilon_1}{1-\varepsilon_1}\leq \frac{\kappa_2(SQ_A)\kappa_2(A) + \varepsilon_1}{1-\varepsilon_1},$$
    and
    $$\kappa_2(A\hat{R}^{-1})\leq \kappa_2(SQ_A)\frac{1 + \varepsilon_1}{1-\varepsilon_1}.$$
\end{lemma}
\begin{proof}
We have
\begin{equation}\label{eq:boundhatRB1}
    \|\hat{R}\|_2 = \|\tilde{Q}\hat{R}\|_2 = \|QR + E\|_2\leq \|R\|_2 + \|E\|_2,
\end{equation}
and
\begin{equation}\label{eq:boundhatRB2}
    \sigma_{\min}(\hat{R}) = \sigma_{\min}(QR + E) \geq \sigma_{\min}(R) - \|E\|_2.
\end{equation}
We see that $\sigma_{\min}(\hat{R})>0$ provided that $\varepsilon_1<1$. We can combine~\cref{eq:boundhatRB1} and~\cref{eq:boundhatRB2} together to obtain an upper bound on $\kappa_2(\hat{R})$. To bound $\kappa_2(A\hat{R}^{-1})$, we use the fact that\footnote{The `$\dagger$' superscript on a matrix denotes the matrix pseudo-inverse.} $A\hat{R}^{-1}\tilde{Q}^T = AR^{-1}(Q+ER^{-1})^{\dagger}$ to find that for any $v \in \mathbb{R}^m$, we have
\begin{equation}
\|v^TAR^{-1}\|_2 (1+\varepsilon_1)^{-1} \leq \|v^TA\hat{R}^{-1}\|_2 \leq \|v^TAR^{-1}\|_2 (1-\varepsilon_1)^{-1}.
\end{equation}
The bound follows as $\kappa_2(AR^{-1}) = \kappa_2(SQ_A)$ (see~\cref{lem:exactcond}).
\end{proof}
\cref{lem:condsRandARinv} shows us, assuming $\varepsilon_1$ is sufficiently small, that $\kappa_2(\hat{R})$ cannot be larger than about $\kappa_2(SQ_A)\kappa_2(A)$, and $\kappa_2(A\hat{R}^{-1})$ cannot be much larger than $\kappa_2(SQ_A)$.

As the following lemma shows, the error incurred by solving the triangular system depends on the quantities in~\cref{lem:condsRandARinv}.
\begin{lemma}\label{lem:condYhatinitial}
    Let $\hat{R}^{-1}$ be the upper triangular matrix defined in~\cref{lem:condsRandARinv} and let $\hat{Y}$ be the matrix obtained by computing $A\hat{R}^{-1}$ with forward substitution in finite precision. Define $\Delta Y = \hat{Y} - A\hat{R}^{-1}$ and $\varepsilon_2 = \|\Delta Y\|_2 \|(A\hat{R}^{-1})^{\dagger}\|_2$. Assuming that $\varepsilon_2 < 1$, we find that $\hat{Y}$ is full rank and
\[
        \kappa_2(\hat{Y}) \leq \frac{\kappa_2(A\hat{R}^{-1}) + \varepsilon_2}{1-\varepsilon_2}.
\]
\end{lemma}
\begin{proof}
    This follows immediately from the relationship $\hat{Y} = A\hat{R}^{-1} + \Delta Y$ and similar reasoning to the proof of~\cref{lem:condsRandARinv}.
\end{proof}
In exact arithmetic, clearly $\kappa_2(Y)=\kappa_2(AR^{-1})$. The numerical error incurred due to the triangular solve is controlled by $\|\Delta Y\|$, which depends on $\kappa_2(\hat{R})$ and $\kappa_2(A\hat{R}^{-1})$. We show that the error is small under the assumption that $\kappa_2(A)u\ll 1$. As a result, the condition number of $\hat{Y}$ will be bounded by a small constant related to the conditioning of $S$ and $SQ_A$. Thus, the computed $AP$ has a reasonable condition number provided that the numerical errors $\varepsilon_1$ and $\varepsilon_2$ are small.
\subsection{Bounding the numerical errors}
We now bound the errors $\varepsilon_1$ and $\varepsilon_2$.
\begin{lemma}\label{lem:boundvarepsilons}
    Assume the same setup as in~\cref{lem:condsRandARinv} and~\cref{lem:condYhatinitial}, where $\varepsilon_1 = \|E\|_2\|R^{-1}\|_2$ and $\varepsilon_2 = \|\Delta Y\|_2\|(A\hat{R}^{-1})^{\dagger}\|_2$, and all quantities in finite precision are computed with unit round-off $u$. Furthermore, define 
\begin{equation}\label{eq:defconstantsC1kX}
C_1 := \sqrt{sn}\gamma_m + \sqrt{n}\tilde{\gamma}_{sn}(1+\sqrt{sn}\gamma_m), \qquad k(S) := \|S\|_2\|(SQ_A)^{\dagger}\|_2.
\end{equation}
If $\kappa_2(A)<1/(C_1k(S))$, then we have
\[
    \varepsilon_1\leq C_1k(S)\kappa_2(A),
\]
and
\[
    \varepsilon_2\leq \frac{n\gamma_n\kappa_2(SQ_A)(\kappa_2(A)\kappa_2(SQ_A) + \varepsilon_1)(1 + \varepsilon_1)}{\left(1 - \varepsilon_1 - \sqrt{n}\gamma_n(\kappa_2(A)\kappa_2(SQ_A) + \varepsilon_1)\right)(1-\varepsilon_1)}.
\]
\end{lemma}
\begin{proof}
    We first note that $E = E_1 + E_2$, where $E_1 = \hat{B} - SA$ is the error arising from matrix-matrix multiplication and $E_2 = \tilde{Q}\hat{R} - \hat{B}$ is the error arising from the Householder QR algorithm.
    We can bound these terms using~\cref{eq:matmatmulterr1} and~\cref{eq:QRerr} to find
    \begin{align*}
        \|E\|_2  &\leq \|E_1\|_2 + \sqrt{n}\tilde{\gamma}_{sn}\|SA + E_1\|_2 \\
        &\leq \sqrt{n}\tilde{\gamma}_{sn}\|S\|_2\|A\|_2 + (1+\sqrt{n}\tilde{\gamma}_{sn})\sqrt{sn}\gamma_m\|S\|_2\|A\|_2  = C_1\|S\|_2\|A\|_2.
    \end{align*}
    This is then combined with 
    \begin{equation*}
        \sigma_{\min}(R) = \sigma_{\min}(QR) = \sigma_{\min}(SQ_AR_A) \geq \sigma_{\min}(SQ_A)\sigma_{\min}(A).
    \end{equation*}
 As for $\|\Delta Y\|_2$, by~\cref{eq:trisubserr} we have
\begin{equation*}
    \|\Delta Y\|_2 \leq \|\Delta Y\|_F\leq \frac{\sqrt{n}\gamma_n\kappa_2(\hat{R})}{1 - \sqrt{n}\gamma_n\kappa_2(\hat{R})}\|A\hat{R}^{-1}\|_F \leq \phi(\hat{R}) \|A\hat{R}^{-1}\|_2,
\end{equation*}
where 
\begin{equation}\label{eq:definitionphiRB}\phi(\hat{R}) := \frac{n\gamma_n\kappa_2(\hat{R})}{1-\sqrt{n}\gamma_n\kappa_2(\hat{R})} \leq \frac{n\gamma_n(\kappa_2(A)\kappa_2(SQ_A) + \varepsilon_1)}{1-\varepsilon_1-\sqrt{n}\gamma_n(\kappa_2(A)\kappa_2(SQ_A) + \varepsilon_1)}.\end{equation}
We have $\varepsilon_2\leq \phi(\hat{R})\kappa_2(A\hat{R}^{-1})$, which can 
be bounded using the result on $\kappa_2(A\hat{R}^{-1})$ in~\cref{lem:condsRandARinv}.
\end{proof}
The condition that $\kappa_2(A) < 1/(C_1k(S))$ is closely related to the assumption $\kappa_2(A) \ll u^{-1}$ (see~\cref{subsec:boundingkappaYhat} for further discussion), which ensures that $\varepsilon_1<\kappa_2(A)u \ll 1$. The bound $\varepsilon_2\leq\phi(\hat{R})\kappa_2(A\hat{R}^{-1})$, where $\phi(\hat{R})$ is defined in~\cref{eq:definitionphiRB}, tells us that the dominant term in the bound for $\varepsilon_2$ is $n^2\kappa_2(SQ_A)\kappa_2(A)u$, when $\varepsilon_1<1$.

\subsection{Bounding $\mathbf{\boldsymbol\kappa_2(\hat{Y})}$}\label{subsec:boundingkappaYhat} We now combine \cref{lem:condsRandARinv,lem:condYhatinitial,lem:boundvarepsilons} to obtain a bound on the condition number of $\hat{Y}$, which is the computed version of $AR^{-1}$. We find that $\kappa_2(\hat{Y})\leq 4\kappa_2(SQ_A) + 1$ provided that $
\kappa_2(A)$ is sufficiently small. 

\begin{theorem}\label{thm:mainTheoremYhat}
    Let $A\in\mathbb{R}^{m\times n}$ have linearly independent columns, $S\in\mathbb{R}^{s\times n}$ have linearly independent rows, and $m>s>n$. Let $A = Q_AR_A$ and $SA = QR$ be the economized QR factorizations of $A$ and $B = SA$, respectively. Assume that
    \begin{equation}\label{eq:assumptionsMainThmYhat}
        \kappa_2(SQ_A) <49 ,\quad n^2u < 1/201, \quad \frac{m}{n}+\sqrt{n} > 200,
    \end{equation}
    and that $\kappa_2(A)$ is sufficiently small so that
    \begin{equation}\label{eq:assumptionKappaAMainThmYhat}
        \kappa_2(A) < \frac{1}{3C_1k(S)},
    \end{equation} 
    where $C_1 = \mathcal{O}((\sqrt{sn}m + sn^{3/2})u)$ and $k(S)$ are defined in~\cref{eq:defconstantsC1kX}. 
    Compute the preconditioned matrix $\hat{Y}$ as detailed in~\cref{lem:condsRandARinv} and~\cref{lem:condYhatinitial}.\footnote{That is, compute the matrix-matrix product $SA$, compute its $R$ factor by the Householder QR algorithm, and finally compute the preconditioned matrix by solving $AR^{-1}$ with forward substitution.} Then,
\[
        \kappa_2(\hat{Y})\leq 4\kappa_2(SQ_A) + 1.
\]
\end{theorem}
\begin{proof} Throughout the proof, we assume the notation introduced in~\cref{lem:condsRandARinv,lem:condYhatinitial,lem:boundvarepsilons}. By assumption \cref{eq:assumptionKappaAMainThmYhat} we immediately see that $\varepsilon_1<1/3$, defined in~\cref{lem:condsRandARinv}, and hence $\hat{R}$ and $A\hat{R}^{-1}$ are full rank. It follows that $\kappa_2(A\hat{R}^{-1})\leq 2\kappa_2(SQ_A)$. Furthermore, also by~\cref{lem:condsRandARinv}, we find 
\begin{equation}\label{eq:boundkappahatRB}
    \kappa_2(\hat{R})\leq \frac{\kappa_2(R) + \varepsilon_1}{1-\varepsilon_1}\leq\frac{1}{2}(3\kappa_2(SQ_A)\kappa_2(A) + 1) < \frac{1}{2}\left(\frac{1}{C_1} + 1\right).
\end{equation}
By using the fact that $m>s>n$, we can show that $C_1 > (m + n^{3/2})\gamma_n,$
which can be substituted into~\cref{eq:boundkappahatRB}. One can use the bound on $\kappa_2(\hat{R})$ to then bound $\phi(\hat{R})$, defined in~\cref{eq:definitionphiRB}, as
$$\phi(\hat{R}) = \frac{n\gamma_n\kappa_2(\hat{R})}{1 - \sqrt{n}\gamma_n\kappa_2(\hat{R})} < \frac{\frac{1}{2}n\left(\frac{1}{m + n^{3/2}} + \gamma_n \right)}{1-\frac{1}{2}\sqrt{n}\left(\frac{1}{m + n^{3/2}} + \gamma_n \right)} = \frac{\sqrt{n}c(m,n)}{1-c(m,n)},$$
where $$c(m,n) = \frac{1}{2}\sqrt{n}\left(\frac{1}{m+n^{3/2}} + \gamma_n\right).$$
The assumptions in~\cref{eq:assumptionsMainThmYhat} imply that 
$c(m,n) < (4\sqrt{n}\kappa_2(SQ_A) + 1)^{-1}$, where we used that $n^2u<1/201$, i.e. $n\gamma_n < 1/200$.
Now we have that
$\phi(\hat{R}) < (4\kappa_2(SQ_A))^{-1}$ and so 
$$\varepsilon_2 \leq \phi(\hat{R})\kappa_2(A\hat{R}^{-1}) \leq \kappa_2(SQ_A)\phi(\hat{R})\frac{1+\varepsilon_1}{1-\varepsilon_1}< \frac{1}{2}.$$
By~\cref{lem:condYhatinitial}, $\hat{Y}$ is full rank and its condition number is bounded by
$$\kappa_2(\hat{Y})\leq\frac{\kappa_2(A\hat{R}^{-1}) + \varepsilon_2}{1-\varepsilon_2} < 4\kappa_2(SQ_A) + 1.$$
\end{proof}
It is worth discussing the assumptions of~\cref{thm:mainTheoremYhat}. Consider the assumption in~\cref{eq:assumptionKappaAMainThmYhat}. The $k(S)$ term (see~\cref{eq:defconstantsC1kX}) is a small constant depending on the particulars of the embedding matrix. The constant $C_1$ (see~\cref{eq:defconstantsC1kX}) can be bounded by the product of a low-degree polynomial in $m$, $n$, and $s$ and the unit-round-off $u$, say $C_1\leq p(m,n,s)u$. We specifically have $C_1 = \mathcal{O}(\sqrt{sn}(m + \sqrt{s}n)u)$, as $\gamma_m = \mathcal{O}(mu)$ and $\tilde{\gamma}_{sn} = \mathcal{O}(snu)$. In classical stability analysis, this polynomial term $p(m,n,s)$ can be more or less ignored. As a result, a condition of the form $\kappa_2(A)<(p(m,n,s)u)^{-1}$ is often loosely restated as $\kappa_2(A) \ll u^{-1}$. 

The rationale behind ignoring these $m$, $n$, and $s$ terms is that classic stability analysis, as we performed in this section, provides us with worst-case bounds on numerical errors. These are generally pessimistic~\cite{Higham2002}; composing $n$ operations would only result in an error of order $nu$ if each rounding error is of the same sign and of maximum magnitude~\cite{higham2019new}. We could improve these bounds by using probabilistic backward error analysis, which would give us results proportional to $\sqrt{n}u$ instead of $nu$~\cite{higham2019new,higham2020sharper}. However, it must be noted that even probabilistic stability analysis would theoretically result in a large factor multiplied by $u$. 

The other assumptions of~\cref{thm:mainTheoremYhat}, in~\cref{eq:assumptionsMainThmYhat}, are reasonable in practice. The condition number of $SQ_A$ is generally observed to be bounded by a modest number, such as $5$ or $10$, with high probability. Furthermore, we assume that the problem dimensions are small compared to the unit round-off. Most importantly, ~\cref{thm:mainTheoremYhat} informs us that, provided that $\kappa_2(A)\ll u^{-1}$, the condition number of the preconditioned matrix is independent of $\kappa_2(A)$ (with high probability).

\subsection{The backward error of sketch-and-apply Blendenpik}\label{sec:theBWerrofSaA}
Until now, we have only considered the condition number of the preconditioned matrix $AP$. The analysis in this section related rounding errors incurred in the final steps --- solving $(AR^{-1})z = b$ with LSQR and computing $x = R^{-1}z$ with back substitution --- to the backward error of sketch-and-apply. It is not immediate that the last step in the algorithm preserves the numerical stability of the algorithm. After all, solving linear systems with $R$ in LSQR iterations is one potential reason for the numerical instability of sketch-and-precondition Blendenpik. Nonetheless, we show that sketch-and-apply Blendenpik performs as well as LSQR on very well-conditioned matrices and that the numerical error in the last step is $\mathcal{O}(u)$ instead of $\mathcal{O}(u\kappa_2(A))$ (see~\cref{thm:mainTheoremLSQR}).

Throughout this analysis, we will not consider a sketch-and-solve initialization and consider stability properties from any starting point.

We show that the backward error in sketch-and-apply Blendenpik is of the same order as the backward error from unpreconditioned LSQR with $\text{fl}(AR^{-1})$, under the assumptions of~\cref{thm:mainTheoremYhat}. 
\begin{theorem}\label{thm:mainTheoremLSQR}
Assume the embedding matrix $S\in\mathbb{R}^{s\times m}$ and $A\in\mathbb{R}^{m\times n}$ satisfy assumptions~\cref{eq:assumptionsMainThmYhat} and~\cref{eq:assumptionKappaAMainThmYhat} in~\cref{thm:mainTheoremYhat}. Apply~\cref{alg:ARBlendenpik} and assume that LSQR terminates at an iterate $\hat{z}$ that satisfies the backward error\footnote{Note that the backward error term $\Delta \hat{Y}$ is not the same as the numerical error $\Delta Y$ resulting from solving the triangular system $Y = A\hat{R}^{-1}$ in finite precision.}
\begin{equation}  \label{eq:berrCGwellcond}
\hat{z} = (\hat{Y} + \Delta   \hat{Y})^{\dagger}(b + \delta b).  
\end{equation}
Then, the computed solution $\hat{x}$ to $Ax = b$ has the backward error $(A + \Delta A)\hat{x} = b + \delta b$
satisfying
\[
    \|\Delta A\|_2 < \|S\|_2\|A\|_2\left(6.04n\gamma_n\|(SQ_A)^{\dagger}\|_2 + 2.01\|\Delta \hat{Y}\|_2 \right).
\]
\end{theorem}
\begin{proof}
The final step of~\cref{alg:ARBlendenpik} involves solving the triangular system $\hat{R}x = \hat{z}$. We know from~\cref{eq:backerrbacksub} that the computed solution $\hat{x}$ satisfies
$$(\hat{R} + \Delta \hat{R})\hat{x} = \hat{z}, \quad \|\Delta \hat{R}\|_2 \leq \gamma_n\sqrt{n} \|\hat{R}\|_2.$$
We now have
\begin{equation*}
    \hat{x} = (\hat{R} + \Delta \hat{R})^{-1}\hat{z}
     = (\hat{R} + \Delta \hat{R})^{-1}(\hat{Y} + \Delta \hat{Y})^{\dagger}(b + \delta b)
     = (A + \Delta A)^{\dagger}(b+ \delta b),
\end{equation*}
where $\Delta A = (\hat{Y}\hat{R}-A) +  \Delta \hat{Y}\hat{R} +\hat{Y}\Delta \hat{R} + \Delta \hat{Y}\Delta \hat{R}.$
For the $\hat{Y}\hat{R} -A$ term  we have that each row satisfies
$$ \hat{y}_i^T\hat{R} - a_i^T = -\hat{y}_i^T\Delta_i\hat{R}, \qquad |\Delta_i \hat{R}|\leq \gamma_n|\hat{R}|,$$
so that $\|\hat{Y}\hat{R}-A\|_2 \leq n\gamma_n\|\hat{R}\|_2\|\hat{Y}\|_2$.
It follows that
\begin{equation*}
    \|\Delta A\|_2 \leq \|\hat{R}\|_2\left( (n+\sqrt{n})\gamma_n\|\hat{Y}\|_2 + (1 + \sqrt{n}\gamma_n)\|\Delta \hat{Y}\|_2\right),
\end{equation*}
where $\|\hat{R}\|_2\leq (1+C_1)\|S\|_2\|A\|_2 < 2\|S\|_2\|A\|_2,$
and
\begin{equation*}
    \|\hat{Y}\|_2 
    \leq\|A\hat{R}^{-1}\|_2 + \|\Delta Y\|_2 \leq \left(1 + \phi(\hat{R})\right)\frac{\|(SQ_A)^{\dagger}\|_2}{1-\varepsilon_1}<1.51\|(SQ_A)^{\dagger}\|_2,
\end{equation*}
where the last inequality follows from the assumptions~\cref{eq:assumptionsMainThmYhat} and~\cref{eq:assumptionKappaAMainThmYhat}. We obtain
$$\|\Delta A\| \leq 2\|S\|_2\|A\|_2\left(3.02n\gamma_n\|(SQ_A)^{\dagger}\|_2 + 1.005\|\Delta \hat{Y}\|_2\right),$$
where we used that $n+\sqrt{n}\leq 2n$ and $\sqrt{n}\gamma_n\leq n\gamma_n\leq 1/200$ by~\cref{eq:assumptionsMainThmYhat}.
\end{proof}
The moral of~\cref{thm:mainTheoremLSQR} is that sketch-and-apply Blendenpik is backward stable as long as LSQR on $\hat{Y}$ is backward stable, that is, $\|\Delta   \hat{Y}\|_2$ and $\|\delta b\|_2/\|b\|_2$ are both $\mathcal{O}(u)$ in~\cref{eq:berrCGwellcond}. 
Under the conditions of~\cref{thm:mainTheoremLSQR}, the matrix $\hat{Y}$ is well-conditioned with high probability. 
The backward stability of LSQR on a well-conditioned matrix is not an undisputedly clear fact. Practical evidence is abundant and several studies~\cite{bjorck1998stability, greenbaum1989behavior,greenbaum1997estimating,greenbaum1992predicting,meurant2006lanczos} have addressed the convergence of the conjugate gradient (CG) method (LSQR is a stable implementation of CG applied to the normal equation), we are unaware of a precise result that proves CG (or LSQR) applied to a well-conditioned system is backward stable. Such stability analysis would depend on the precise implementation of the algorithm (see~\cref{sec:discussion}).

Remarkably, the final step of the algorithm does not influence the accuracy of the solution, i.e., solving $Rx = y$ with back substitution. 
This is despite the fact that $R$ is ill-conditioned when $A$ is.
The mechanism with which stability is established is similar to the backward stability of an algorithm based on repeated CholeskyQR~\cite{Yamamoto2015}.

\section{On the numerical stability of smoothed sketch-and-apply Blendenpik}\label{sec:SmoothedAnalysisMain}
We now analyze smoothed sketch-and-apply Blendenpik (see~\cref{alg:smoothedARBlendenpikshort}). When $\kappa_2(A) \ll 1/u$,~\cref{thm:mainTheoremYhat} states that sketch-and-apply computes a preconditioned matrix with a condition number independent of $\kappa_2(A)$ with high probability. Moreover, we can reliably estimate $\kappa_2(A)$ by $\kappa_2(SA)$. We do not have a guarantee on the stability of sketch-and-apply Blendenpik if $\kappa_2(A)\gtrsim 1/u$. It is in this context that we consider smoothing.

One could consider~\cref{alg:smoothedARBlendenpikshort} from two points of view. Firstly, a practitioner may
add an $\mathcal{O}(u\|A\|_2)$ Gaussian perturbation to a severely ill-conditioned problem. As a result, the perturbed problem can be accurately solved with sketch-and-apply. Secondly, floating point errors arising from the representation of any matrix in finite precision could be modeled as a Gaussian perturbation and then smoothing partially explains the success of sketch-and-apply to highly ill-conditioned LS problems.

\subsection{Bounding $\mathbf{\boldsymbol\kappa_2(\hat{Y})}$}\label{sec:boundyhatsmoothed}
Similar to the stability analysis in~\cref{sec:condition}, we have the following steps in finite precision after we have smoothed $A$:
\begin{enumerate}
    \item $\tilde{A} = A + \sigma G/\sqrt{m}$,
    \item $\hat{B}=S\tilde{A} + E_1$,
    \item $\tilde{Q}\hat{R} = \hat{B} + E_2$,
    \item $\hat{Y} = \tilde{A}\hat{R}^{-1} + \Delta Y$,
\end{enumerate}
where the hats denote that we are accounting for rounding errors and the scaling parameter $\sigma$ controls how much $A$ is perturbed. Assuming $\sigma$ is sufficiently large to ensure that $\tilde{A}$ satisfies~\cref{eq:assumptionKappaAMainThmYhat}, we then show that the stability analysis from~\cref{sec:condition} carries over to $\tilde{A}$.

The following theorem tells us how large $\sigma$ needs to be.
\begin{theorem}\label{thm:smoothedfiniteprec} 
Let $A\in\mathbb{R}^{m\times n}$ with $m\geq 3n$, $S\in\mathbb{R}^{s\times n}$ have linearly independent rows, and let $m>s>n$. Set 
\begin{equation}\label{eq:magnitudesigma}
    \tilde{A} = A +\sigma G/\sqrt{m}, \qquad \sigma = 52\|A\|_2k(S)sm\sqrt{n}u,
\end{equation}
where $G\in\mathbb{R}^{m\times n}$ is a standard Gaussian matrix with i.i.d.~entries, $k(S)$ is defined in~\cref{eq:defconstantsC1kX} and $u$ is the unit round-off.
Let $\tilde{A} = Q_{\tilde{A}}R_{\tilde{A}}$ and $S\tilde{A} = QR$ be economized QR factorizations of $\tilde{A}$ and $B = S\tilde{A}$, respectively. Lastly, assume that\footnote{We also assume that $cn<m$ for the same integer $c$ defined in $\tilde{\gamma}_{sn}$ to simplify the notation.} 
\begin{equation}\label{eq:assumptionsSmoothedThmYhat}
        \kappa_2(SQ_{\tilde{A}}) <49,\qquad  \frac{m}{n}+\sqrt{n} > 200,  \qquad smu<1/67.
    \end{equation}
Now compute first the product $S\tilde{A}$, then the upper triangular factor $R$ by the Householder QR algorithm on $S\tilde{A}$, and finally compute $\hat{Y} = \tilde{A}R^{-1}$ with forward substitution. Assume all computations are performed in finite precision. Then,
\[
    \mathbb{P}\left(\kappa_2(\hat{Y}) > 4\kappa_2(SQ_{\tilde{A}}) + 1\right) < 2^{n-m}.
\]
\end{theorem}
\begin{proof}
We start by showing that $\sigma$ in~\cref{eq:magnitudesigma} is sufficiently large to ensure that $\tilde{A}$ satisfies~\cref{eq:assumptionKappaAMainThmYhat} in~\cref{thm:mainTheoremYhat}. To this end, note the assumptions in~\cref{eq:assumptionsSmoothedThmYhat} imply that $mu<1/64$ and $csnu<1/64$ so that we have $\gamma_m<1.02mu$ and $\tilde{\gamma}_{sn}<1.02csnu<1.02smu$. We can then bound $C_1$ defined in~\cref{eq:defconstantsC1kX} as
$$C_1 < 1.02\times65sm(1 + \sqrt{n})u/64 < 2.1 sm\sqrt{n}u.$$
As a result
\[
    \mathbb{P}\left(\kappa_2(\tilde{A})>\frac{1}{3C_1k(S)}  \right) < \mathbb{P}\left(\kappa_2(\tilde{A}) > \frac{8.25}{52k(S)sm\sqrt{n}u} \right) < 2^{n-m}.
\]
The final result on $\kappa_2(\hat{Y})$ follows by applying~\cref{thm:mainTheoremYhat} to $\tilde{A}$, which is possible since the last assumption in~\cref{eq:assumptionsSmoothedThmYhat} is stronger than $nu<1/201$.
\end{proof}

\cref{thm:smoothedfiniteprec} shows that the magnitude of the perturbation $\sigma G/\sqrt{m}$ needs to be approximately $\mathcal{O}(sm\sqrt{n}u)$ to ensure the condition number of $\tilde{A}$ is sufficiently small, that is, $\kappa_2(A)<1/u$. The low-degree polynomial term $sm\sqrt{n}$ arises from~\cref{eq:assumptionKappaAMainThmYhat}, which includes the constant $C_1 = \mathcal{O}(\sqrt{sn}(m + \sqrt{s}n)u)$. As discussed in~\cref{sec:theBWerrofSaA}, these are often pessimistic bounds. Probabilistic error analysis provides a theoretical justification to select $\sigma$ proportional to $\sqrt{sm}n^{1/4}u$. This may still be a relatively large perturbation if the dimensions of the problem are enormous, enlarging the backward error to an unacceptable level (see~\cref{sec:smoothedanalysisLSQR}).

\subsection{The backward error of smoothed sketch-and-apply Blendenpik} \label{sec:smoothedanalysisLSQR}
In~\cref{sec:boundyhatsmoothed}, we showed that a small additive random perturbation to a small-skinny matrix ensures that the condition number of the perturbed matrix is sufficiently small. This allows us to conclude that $\kappa_2(\hat{Y})$ is small with high probability.  We now show that smoothed sketch-and-apply Blendenpik (see~\cref{alg:smoothedARBlendenpikshort}) computes a backward error similar to~\cref{thm:mainTheoremLSQR} with an additional term from the additive perturbation.

\begin{theorem}\label{thm:smoothedLSQR}
    Assume the same setup as in~\cref{thm:smoothedfiniteprec}. Apply~\cref{alg:ARBlendenpik} to $\tilde{A}x=b$ and assume LSQR terminates at an iterate $\hat{z}$ that satisfies the backward error
    \begin{equation*}
        \hat{z} = (\hat{Y} + \Delta \hat{Y})^{\dagger}(b+\delta b).
    \end{equation*}
    Then, the computed solution $\hat{x}$ to $Ax=b$ has the backward error $(A + \Delta A)\hat{x} = b + \delta b$ satisfying
    \begin{multline}
        \|\Delta A\|_2 \leq \|A\|_2\left[\vphantom{\frac12}\|S\|_2\left(1 + 197k(S)sm\sqrt{n}u\right)\right.\\ 
        \left. \left(6.04n\gamma_n\|(SQ_{\tilde{A}})^{\dagger}\|_2 + 2.01\|\Delta \hat{Y}\|_2 \right) + 197k(S)sm\sqrt{n}u\vphantom{\frac12}\right]
    \end{multline}
    with probability at least $1-2^{n-m}-10^{-3}$.
\end{theorem}
\begin{proof}
    We combine~\cref{thm:mainTheoremLSQR} and~\cref{thm:smoothedfiniteprec} to obtain the result. Additionally, we use the following classic result in the analysis of Gaussian matrices~\cite{davidson2001local}:
    \begin{equation*}
        \mathbb{P}\left(\|G\|_2/\sqrt{m} > 1 + \sqrt{\frac{n}{m}} + \frac{t}{\sqrt{m}}\right)\leq e^{-t^2/2},
    \end{equation*}
    applied with $t = 3.8$. This shows $\|\tilde{A}\|_2\leq\|A\|_2 + \sigma\|G\|_2/\sqrt{m}\leq \|A\|_2 + 3.78\sigma$ with probability at least $10^{-3}$, where $\sigma = 52\|A\|_2k(S)sm\sqrt{n}u$.
\end{proof}

\cref{thm:smoothedLSQR} allows us to conclude that under mild assumptions, smoothed sketch-and-apply Blendenpik results in a backward error of the same order as the backward error of unpreconditioned LSQR on $\hat{Y}z = b$. Under the conditions of the theorem, $\hat{Y}$ is indeed a well-conditioned matrix and LSQR is expected to obtain backward stable solutions. We removed the assumptions on $\kappa_2(A)$. 

\Cref{thm:smoothedLSQR} suggests that the perturbation needs to be of size $\mathcal{O}(sm\sqrt{n}u)$. This would be impermissible for many problems, especially if the dimensions are large.  We recommend a perturbation of magnitude approximately $10\|A\|_2u$. Although we lack the theoretical evidence for this choice, it works in our experiments (see~\cref{sec:numexpsmoothed}). 

Smoothing an LS problem should be done with much caution. There are many contexts where $\kappa_2(A) > u^{-1}$, yet perturbing the problem is unnecessary. For instance, if the ill-conditioning is caused by poor column scaling, sketch-and-apply will still compute accurate solutions without smoothing (see~\cref{sec:discussion}).
In this sense, $\kappa_2(A)$ is not an effective predictor of whether smoothing will be beneficial. We suggest one should smooth only if sketch-and-apply does not converge (see~\cref{alg:master}).

\section{Master sketch-and-apply algorithm and numerical experiments}\label{sec:masteralgorithmandnumexp}
We now present the implementation of the sketch-and-apply algorithms (see~\cref{alg:master}) in a similar fashion to the overall solver in Blendenpik~\cite{Avron2010a}. We perform the standard initial steps, i.e., drawing an embedding matrix $S$, computing the sketch $SA$, and computing the QR factorization $QR = SA$ using the Householder QR algorithm. Next, we compute the preconditioned matrix $Y = AR^{-1}$ explicitly with forward substitution and use (unpreconditioned) LSQR to solve $Yz = b$. We include sketch-and-solve initialization in this algorithm to speed up convergence. This step consists of computing $z_0 = Q^TSb$. If LSQR converges to the desired tolerance, we compute $x = R^{-1}z$ with back substitution.  In this case, we assume our algorithm computed a backward stable solution.

In the case where LSQR does not converge, we perturb/smooth $A$. We compute $\tilde{A} = A + \sigma G/\sqrt{m}$ for a Gaussian matrix $G$ with $\sigma = 10\|A\|_2u$ and perform the sketch-and-apply steps including a sketch-and-solve initial guess on $\tilde{A}x = b$. We recommend doing this because there are examples, although rare, where sketch-and-apply Blendenpik converges to an accurate solution only after smoothing (see~\cref{sec:numexpsmoothed}). The spectral norm of $A$ can be estimated, for instance, with the sketch $SA$ or using the power method.

\begin{algorithm}[h!]
	\caption{Master algorithm for sketch-and-apply to solve~\cref{eq:LeastSquaresProblem}. Here, HHQR refers to the Householder QR algorithm.} \label{alg:master}
	\begin{algorithmic}[1] 
	\STATE Draw a random sketching matrix $S\in\mathbb{R}^{s\times m}$, where $m\gg s > n$ 
	\STATE Compute $B=SA$ and $c = Sb$
	\STATE 
    Compute both $Q$ and $R$ of the QR factorization of $B$ using HHQR
        \STATE Compute $Y = AR^{-1}$ with forward substitution 
        \STATE Compute initial guess $z_0 = Q^Tc$
	
	\STATE Solve $Yz=b$ with LSQR and no preconditioner and initial guess $z_0$
	    \IF{LSQR converges to the desired tolerance}
                 \STATE Compute $x = R^{-1}z$ with back substitution 
	    \ELSE
                \STATE Draw a random standard Gaussian matrix $G\in\mathbb{R}^{m\times n}$ with i.i.d.~entries
	        \STATE Compute $\tilde{A} = A + \sigma G/\sqrt{m}$ for $\sigma = 10\|A\|_2u$
	        \STATE Compute $B=S\tilde{A}$ 
        	\STATE Compute both $Q$ and $R$ of the QR factorization of $B$ using HHQR
        	\STATE Compute $Y = \tilde{A}R^{-1}$ with forward substitution 
         \STATE Compute initial guess $z_0 = Q^Tc$
	        \STATE Solve $Yz=b$ with LSQR and no preconditioner and initial guess $z_0$
	        \STATE Compute $x = R^{-1}z$ with back substitution 
	    \ENDIF
	\end{algorithmic}
\end{algorithm}

It should be noted that without smoothing, LSQR iteration can converge to a solution with a sub-optimal residual (see~\cref{fig:smoothedfigure} (right)). However, it is infeasible to compute the backward error when $m$ is large as it requires computing the smallest singular value of an $m\times (m+n)$ matrix (although it is possible to reduce the cost~\cite{hallman2020estimating}). 
If the residual obtained from sketch-and-apply is less accurate than desired, one could estimate the condition number of $A$ via the condition number of $R$. If the condition number estimate of $R$ is of the same order of magnitude as $u^{-1}$, then the problem is severely ill-conditioned and one could try smoothing to potentially improve the accuracy of the computed solution. Smoothing can improve the convergence rate of LSQR, even if sketch-and-apply obtains a backward stable solution without smoothing (see~\cref{fig:smoothedfigure_Haar}). Nonetheless one should be careful with smoothing because examples exist where $\kappa_2(A)$ is large but smoothing remains unnecessary (see~\cref{fig:smoothedfigure_Haar}). Exactly when and when not to smooth remains unclear. 

\subsection{Numerical experiments with sketch-and-apply Blendenpik}\label{sec:numexpapply}
Now we examine the performance of sketch-and-precondition Blendenpik with sketch-and-solve initialization, sketch-and-apply Blendenpik with and without sketch-and-solve initialization, and the QR-based direct solver. The experiment considers ill-conditioned LS problems that are randomly generated (see~\cref{sec:instability} for details on how $A$ and $b$ are formed) and can be stored in local cache. The QR solver is implemented with the {\tt qr} and backslash {\tt $\backslash$} (with $R$) commands in MATLAB. The execution times are shown in~\cref{fig:small_time_exp}. 

We find that on this scale, sketch-and-apply Blendenpik, especially with SAS, is competitive relative to QR in terms of computational time. A sketch-and-solve initialization improves the computing time as fewer iterations are needed. Sketch-and-precondition with SAS is the computationally most efficient, and will in almost all cases converge to a solution with optimal residual and backward error. For most practical applications, this should remain the preferred option.
\begin{figure}[h!]
    \centering
    \includegraphics[width = \linewidth]{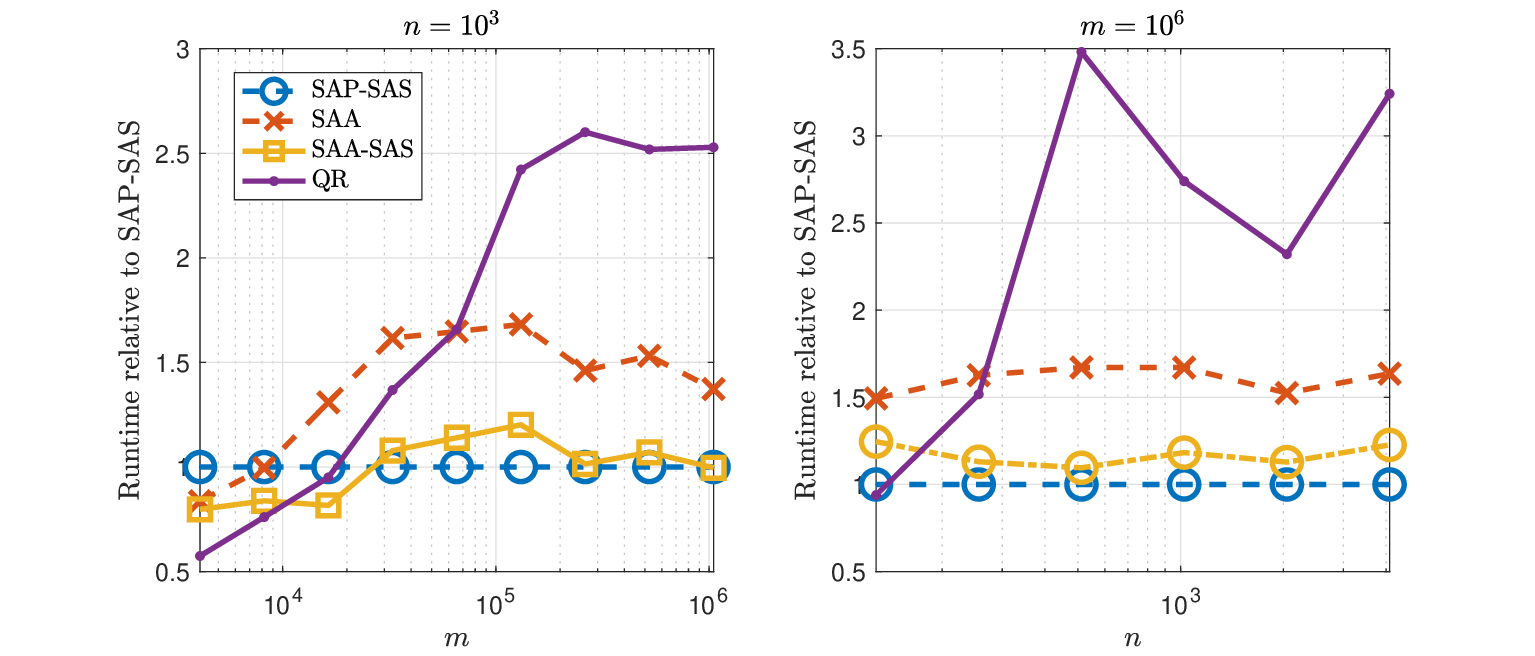}
    \caption{The relative execution time to solve large LS problems using sketch-and-precondition Blendenpik with sketch-and-solve initialization (SAS), sketch-and-apply Blendenpik, sketch-and-apply Blendenpik with SAS, and Householder QR. For each problem $\kappa(A) = 10^{10}$ and the noise level $\|e\|_2 = 10^{-10}$. The tolerance is set to $10^{-12}$ and the maximum number of iterations to 100 (this is never reached). Left: Timings for $n = 4000$ and $2^{12}\leq m \leq 2^{20}$. Right: Timings for $m = 10^6$ and $2^7\leq n\leq 2^{12}$. }
    \label{fig:small_time_exp}
\end{figure}

We furthermore note that sketch-and-apply (and sketch-and-precondition) algorithms could outperform the direct QR method by a larger factor if it is expensive to communicate with the matrix $A$. This context occurs, for example, when $A$ is too large to be stored in local cache and is instead stored on disk. Provided we obtain a good sketch of $A$, the number of iterations in LSQR before convergence will be modest, requiring limited streaming. The QR method, however, requires $n$ views for an $m\times n$ matrix.

\subsection{Numerical experiments with smoothed sketch-and-apply Blendenpik}\label{sec:numexpsmoothed}
Sketch-and-precondition Blendenpik with initialization, sketch-and-apply Blendenpik, its smoothed version, and Householder QR can compute accurate solutions to LS problems, even when $\kappa_2(A) > 1/u$ (see~\cref{fig:smoothedfigure_Haar}, the problems are generated as explained in~\cref{sec:instability}). Therefore, smoothing is only sometimes required for extremely ill-conditioned least squares problems. However, even when smoothing is unnecessary, there is some potential benefit to smoothing because the perturbed LS problem can lead to rapid LSQR convergence. Of course, there is a trade-off here as one converges rapidly to an accurate solution of the perturbed problem, and the computed solution may be less accurate for the original LS problem of interest. 

\begin{figure}[h!]
\centering
\begin{subfigure}{.5\textwidth}
  \centering
\includegraphics[width=\linewidth]{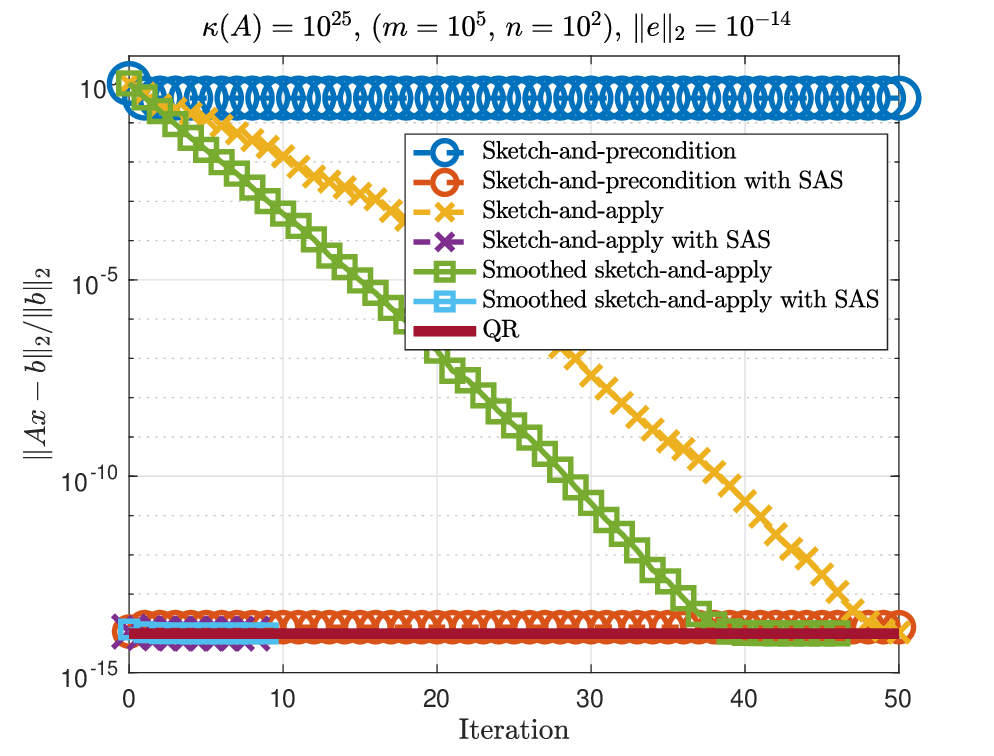}
\end{subfigure}%
\begin{subfigure}{.5\textwidth}
  \centering
  \includegraphics[width= \linewidth]{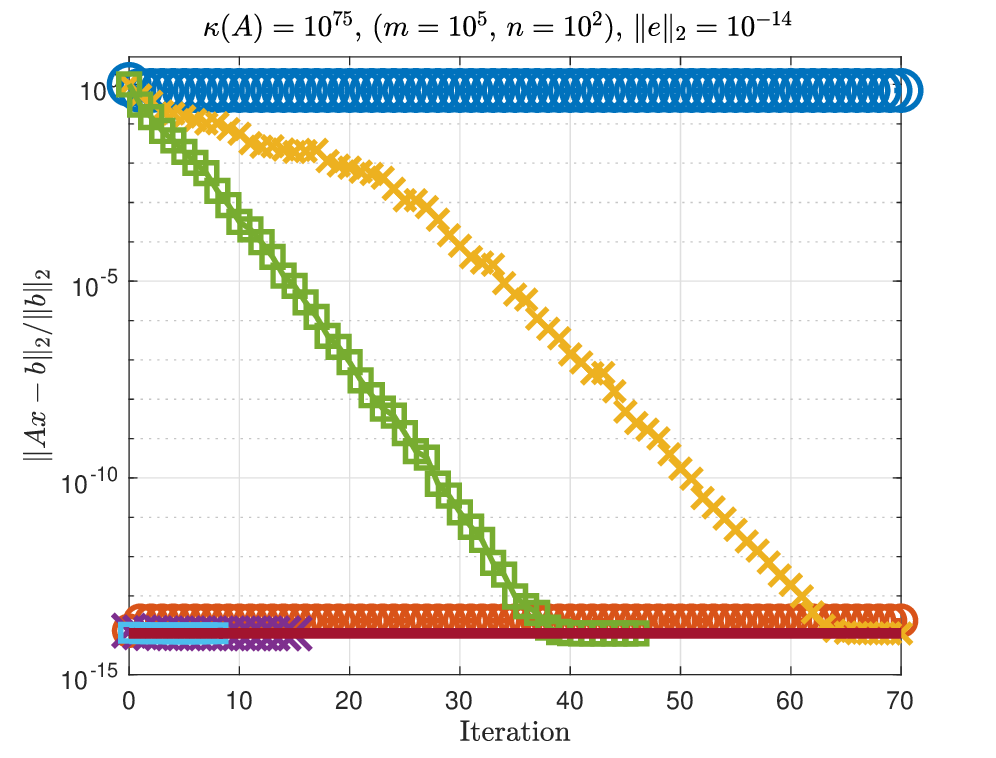}
\end{subfigure}
\caption{Smoothing is not always needed for LS problems with $\kappa_2(A)>1/u$. These matrices are formed by drawing singular vectors from the Haar distribution and letting the singular values decay exponentially from 1 to $1/\kappa_2(A)$. For comparison, we show the residual error computed by Householder QR. These problems are so ill-conditioned that the backward error is not an informative measure, i.e., most computed solutions (e.g., with very few LSQR iterations) give a backward error $\mathcal{O}(u)$.}
\label{fig:smoothedfigure_Haar}
\end{figure}

However, there are also examples for which smoothing provides a more accurate solution to the original problem of $\min_x\|Ax-b\|$, not $\min_x\|\tilde{A}x-b\|$ (see~\cref{fig:smoothedfigure}). For example, we take an LS problem involving a $1000\times 100$ Kahan matrix with $\theta = 1.1$ from the MATLAB \texttt{gallery} collection and another one involving a column-scaled $1000\times 10$ Vandermonde matrix involving equally spaced points between $-1$ and $1$. The Kahan and Vandermonde matrices are designed so that their condition numbers are $>1/u$. 
Without smoothing, even the QR-based algorithm fails to compute an accurate solution along with all the other methods. This is because the problem is so ill-conditioned that even a backward stable solution can behave wildly. However, after smoothing, sketch-and-apply Blendenpik computes an accurate solution to the original LS problems. It should be noted that one could also smooth before solving the problems with QR to obtain more accurate solutions.

\begin{figure}[h!]
\centering
\begin{subfigure}{.5\textwidth}
  \centering
\includegraphics[width=\linewidth]{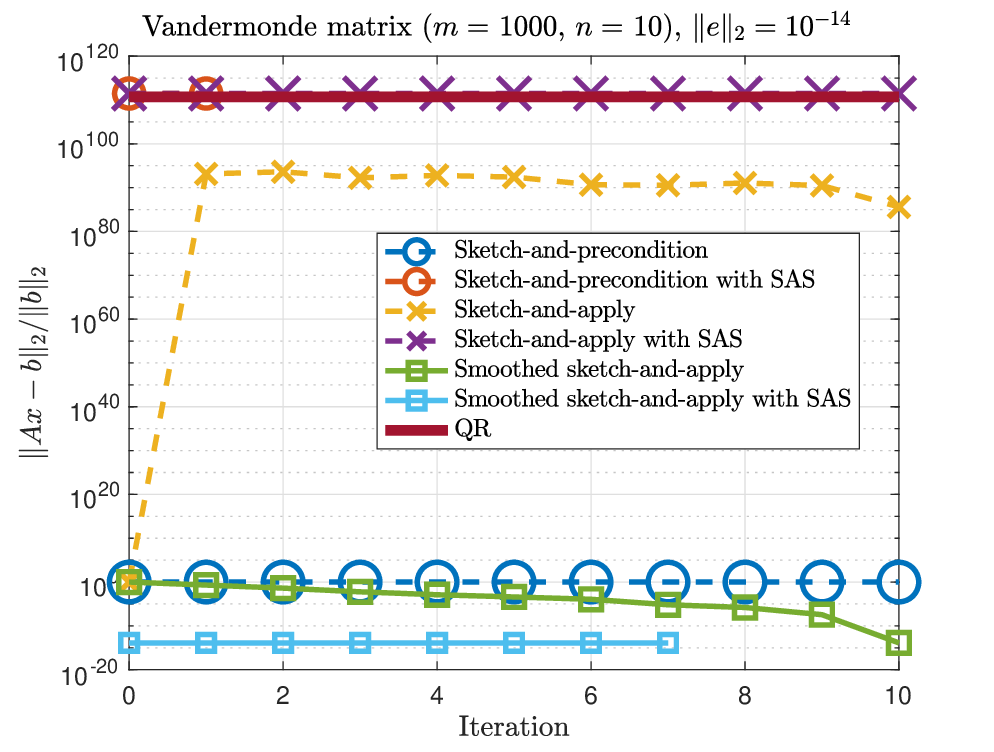}
\end{subfigure}%
\begin{subfigure}{.5\textwidth}
  \centering
  \includegraphics[width= \linewidth]{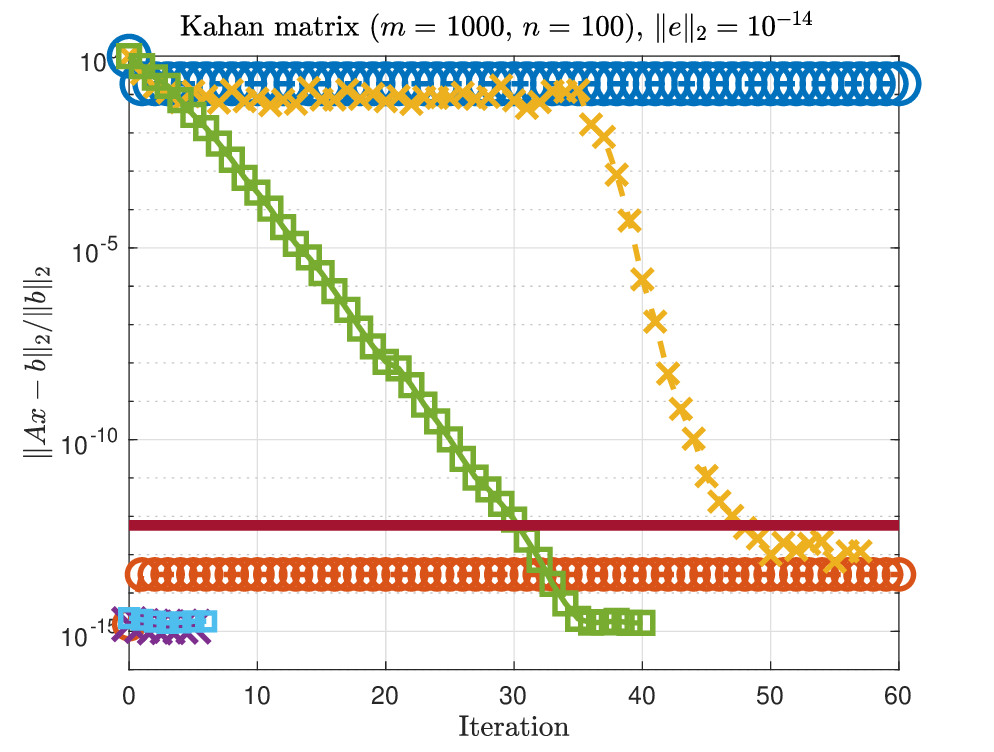}
\end{subfigure}
\caption{Smoothing can solve extremely ill-conditioned LS problems, even when sketch-and-apply Blendenpik and Householder QR cannot. For the Vandermonde matrix, we can only obtain accurate solutions using a smoothed algorithm (note the y-axis). The problem with the Kahan matrix can only be solved optimally with certain sketch-and-apply variations; the initialized sketch-and-precondition algorithm and standard sketch-and-apply attain sub-optimal residuals.
Again, these problems are too ill-conditioned for the backward error to be informative.
}
\label{fig:smoothedfigure}
\end{figure}
\subsection{Sparse matrices}
We briefly discuss the performance of the various algorithms in question on sparse matrices. Firstly, it should be noted that any sketch-and-apply variant will not respect the sparsity: the matrix $Y = AR^{-1}$ is generally dense. As a result, convergence will be slower and computational cost will be unnecessarily high. Remarkably,~\cref{fig:sparsefigures} shows that standard sketch-and-precondition can be numerically stable for sparse matrices with few non-zero entries. It appears that if $\text{nnz}(A)$ is sufficiently small, the rounding errors compound less and the algorithm finds accurate solutions. A precise explanation is left for future work. Unsurprisingly, all algorithms with sketch-and-solve initialization converge rapidly to accurate solutions.
\begin{figure}[h!]
\centering
\begin{subfigure}{.5\textwidth}
  \centering
\includegraphics[width=\linewidth]{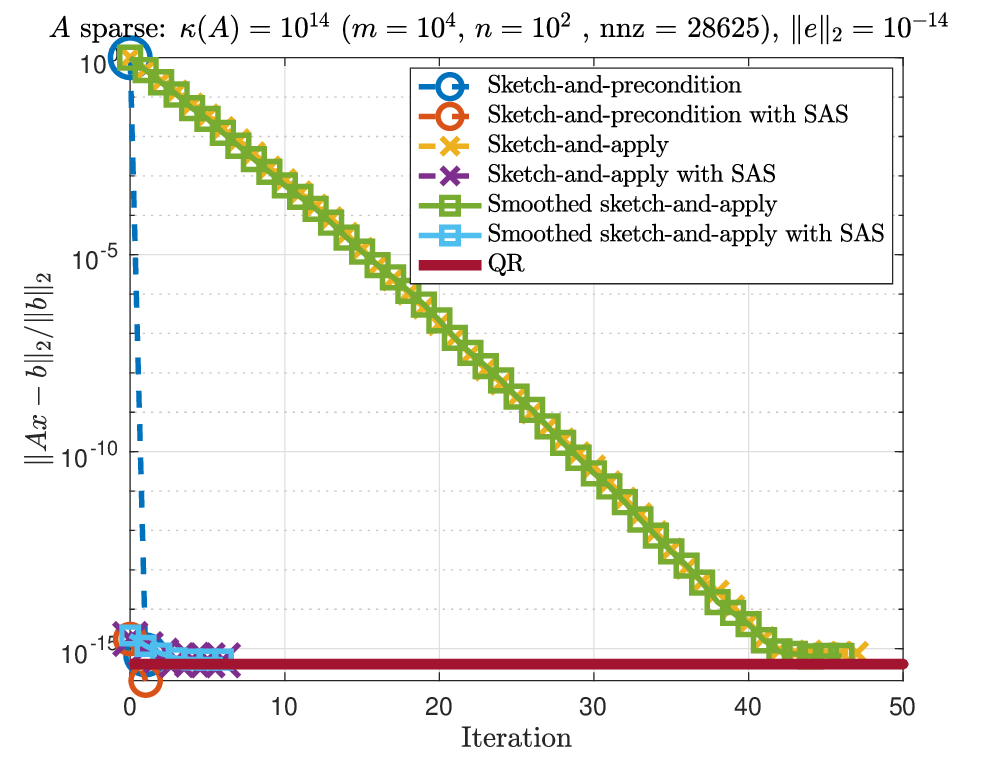}
\end{subfigure}%
\begin{subfigure}{.5\textwidth}
  \centering
  \includegraphics[width= \linewidth]{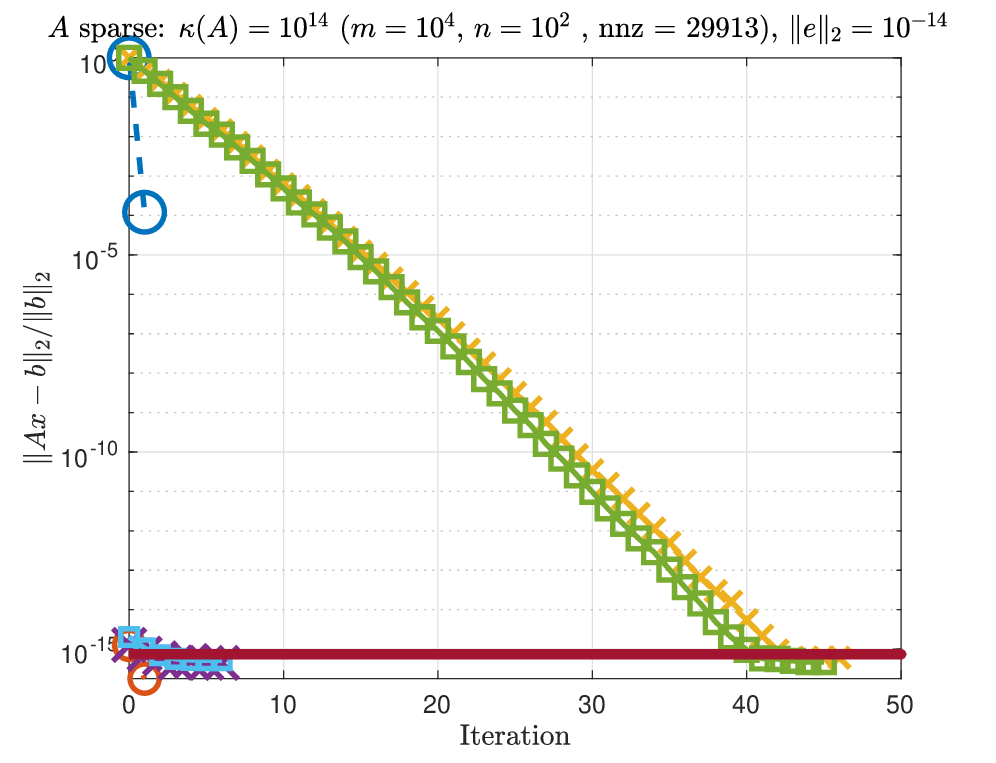}
\end{subfigure}
\caption{Sketch-and-apply algorithms cannot take advantage of sparsity of $A$. Note that sketch-and-precondition without sketch-and-solve initialization can lead to accurate solutions for sparse matrices with few non-zero entries (left figure). 
This behavior depends nontrivially on the sparsity etc; in the right figure, we observe the numerical instabilities we have seen throughout this work.}
\label{fig:sparsefigures}
\end{figure}

\section{Discussion}\label{sec:discussion}
We have shown that sketch-and-precondition algorithms, such as Blendenpik~\cite{Avron2010a}, are numerically unstable in their standard form for solving LS problems. We have stabilized the algorithm by explicitly computing the preconditioned matrix $AP$ and using an unpreconditioned iterative solver on $AP$. We coined this modification sketch-and-apply. We furthermore displayed that using sketch-and-solve initialization greatly improves convergence properties as well as the maximal attainable accuracy of sketch-and-precondition.

\subsection{The effectiveness of sketch-and-solve initialization}
Although a large part of this work was dedicated to investigating the (provable) numerical stability of sketch-and-apply and its smoothed version, one of the main messages---especially for practitioners---should be the remarkable effectiveness of the sketch-and-solve initial guess. Apart from extremely ill-conditioned cases (see~\cref{fig:smoothedfigure}), sketch-and-precondition with a sketch-and-solve initial guess attains accurate solutions in terms of residual, albeit not always backward stable. We urge practitioners to always choose~\cref{alg:blendenpik_initialguess} over~\cref{alg:blendenpik}, as the additional cost is minimal but it results in a better rate of convergence and better maximal attainable accuracy.
\subsection{The numerical stability of data-driven preconditioners}
The numerical instabilities observed in sketch-and-precondition Blendenpik raise larger questions on the stability of iterative methods with data-driven preconditioners for LS problems. Here, we refer to a data-driven preconditioner as a preconditioner constructed directly from $A$, without knowing where $A$ came from (such as the discretization of a continuous problem).~\cref{fig:LSQRcompintro} shows that even using $R_A$ (where $A = Q_AR_A$) as a preconditioner---the perfect data-driven preconditioner---does not lead to a backward stable solution. We suspect this is due to the numerical errors incurred each time the ill-conditioned preconditioner is applied in an iterative solver. Is it possible for a data-driven preconditioner to avoid compounding these rounding errors?

\subsection{The numerical stability of iterative least squares solvers} 
As to the numerical stability of sketch-and-apply, we were not able to state that the backward error $\|\Delta A\|_2$ is $\mathcal{O}(u\|A\|_2)$ (see~\cref{thm:mainTheoremLSQR,thm:smoothedLSQR}). Instead, we have shown that $\|\Delta A\|_2/\|A\|_2$ is of the same order as the backward error $\|\Delta \hat{Y}\|$ incurred when $\hat{Y}z = b$ is solved with unpreconditioned LSQR, where $\hat{Y}$ is well-conditioned. It has proven challenging to understand the literature on the numerical stability of CG-like iterative solvers such as LSQR. Various works seem to strongly hint at backward stability under assumptions on the condition number, but use computational results to complement the claim, see~\cite{bjorck1998stability,greenbaum1997estimating}. 
The numerical stability depends strongly on the specific implementation used in a way that we are yet to understand fully. We note that a recent result by Musco, Musco and Sidford~\cite[Thm.~2.1]{musco2018stability}, when specialized to well-conditioned positive definite linear systems, implies that $\epsilon$ forward error is achieved by using $\mathcal{O}(\log\frac{1}{\epsilon})$ bits, with Lanczos with modified Gram-Schmidt orthogonalization. For well-conditioned linear systems, taking $\epsilon=u$ this would imply the solution has an $\mathcal{O}(u)$ backward error, hence backward stable.

\subsection{Variants of sketch-and-precondition}
We have assumed specific choices for how the randomized preconditioner is constructed. In LSRN~\cite{Meng2014}, for instance, the preconditioner is chosen to be $P=V\Sigma^{-1}$, where $SA = U\Sigma V^T$ is the SVD of the sketch. Numerical experiments show us that SVD-based sketch-and-precondition methods incur similar numerical instabilities as the QR-based variant. The sketch-and-apply technique can also be used with LSRN ideas where the preconditioned matrix is computed as $Y = AV\Sigma^{-1}$. We suspect one can also prove similar results to our sketch-and-apply Blendenpik analysis for sketch-and-apply LSRN. Versions of \cref{lem:condsRandARinv,lem:condYhatinitial} hold almost identically for LSRN, with $R$ replaced by $\Sigma V$. The key difference is the numerical errors $\varepsilon_1$ and $\varepsilon_2$, which we have not investigated. Since computing the SVD is typically more expensive than a QR factorization, we recommend sketch-and-apply Blendenpik. One notable exception is when solving a sequence of Tikhonov regularized LS problems for various regularization parameters~\cite{meier2022randomized, Meng2014}.

\subsection{Ill-conditioning caused by poor column scaling}
If the ill-conditioning in an LS problem is caused by poor column scaling in $A$, we strongly recommend using QR-based sketch-and-apply techniques. The reason is that computing $AR^{-1}$ is invariant to column scaling, while the SVD is not. In fact, the assumption that $\kappa_2(A) \ll u^{-1}$ in~\cref{sec:condition} can be replaced by
$$\min_D\{\kappa_2(AD)\,:\, D = \text{diag($ b^{k_i}$), $k_i\in\mathbb{Z}, \,i = 1,\dots, n$}\} \ll u^{-1},$$
where $b$ is the machine base (usually $b=2$). 
Of course, one can also pre-process an LS problem by scaling the columns of $A$ by powers of the machine base so that each column has a norm that is close to $1$. 

\section*{Acknowledgments}
We thank Erin Carson and Zden\v{e}k Strako\v{s} for their valuable comments regarding the numerical stability of LSQR. We thank Fran\c{c}oise Tisseur for her input on Blendenpik in finite precision arithmetic. We are indebted to Ilse Ipsen and Michael Mahoney for their presentation and subsequent discussions on LS problems, which occurred during the ``Complexity of Matrix Computations" seminar on 1st September 2021. We thank the referees and the editor for their valuable comments. We are especially indebted to the referee who suggested trying sketch-and-precondition with the sketch-and-solve initial guess.

\bibliographystyle{plain}
\bibliography{references.bib}

\begin{thebibliography}{10}

\bibitem{Avron2010a}
H.~Avron, P.~Maymounkov, and S.~Toledo.
\newblock {Blendenpik: Supercharging LAPACK's least-squares solver}.
\newblock {\em SIAM J. Sci. Comput.}, 32(3):1217--1236, 1 2010.

\bibitem{bjorck1996numerical}
{\AA}.~Bj{\"o}rck.
\newblock {\em {Numerical Methods for Least Squares Problems}}.
\newblock SIAM, 1996.

\bibitem{bjorck1998stability}
\r{A}. Bj\"{o}rck, T.~Elfving, and Z.~Strako\v{s}.
\newblock {Stability of conjugate gradient and Lanczos methods for linear least
  squares problems}.
\newblock {\em SIAM J. Mat. Anal. Appl.}, 19(3):720--736, 1998.

\bibitem{boutsidis2009random}
C.~Boutsidis and P.~Drineas.
\newblock Random projections for the nonnegative least-squares problem.
\newblock {\em Lin. Alg. Appl.}, 431(5-7):760--771, 2009.

\bibitem{burgisser2010smoothed}
P.~B{\"u}rgisser and F.~Cucker.
\newblock {Smoothed analysis of Moore--Penrose inversion}.
\newblock {\em SIAM J. Mat. Anal. Appl.}, 31(5):2769--2783, 2010.

\bibitem{cartis2021hashing}
C.~Cartis, J.~Fiala, and Z.~Shao.
\newblock Hashing embeddings of optimal dimension, with applications to linear
  least squares.
\newblock {\em arXiv:2105.11815}, 2021.

\bibitem{chang2009stopping}
X.-W. Chang, C.~C. Paige, and D.~Titley-P{\'e}loquin.
\newblock Stopping criteria for the iterative solution of linear least squares
  problems.
\newblock {\em SIAM J. Mat. Anal. Appl.}, 31(2):831--852, 2009.

\bibitem{clarkson2017low}
K.~L. Clarkson and D.~P. Woodruff.
\newblock Low-rank approximation and regression in input sparsity time.
\newblock {\em J. ACM}, 63(6):1--45, 2017.

\bibitem{davidson2001local}
K.~R. Davidson and S.~J. Szarek.
\newblock {Local operator theory, random matrices and Banach spaces}.
\newblock {\em {Handbook of the geometry of Banach spaces}}, 1(317-366):131,
  2001.

\bibitem{greenbaum1989behavior}
A.~Greenbaum.
\newblock Behavior of slightly perturbed {L}anczos and conjugate-gradient
  recurrences.
\newblock {\em Lin. Alg. Appl.}, 113:7--63, 1989.

\bibitem{greenbaum1997estimating}
A.~Greenbaum.
\newblock Estimating the attainable accuracy of recursively computed residual
  methods.
\newblock {\em SIAM J. Mat. Anal. Appl.}, 18(3):535--551, 1997.

\bibitem{greenbaum1992predicting}
A.~Greenbaum and Z.~Strakos.
\newblock Predicting the behavior of finite precision {L}anczos and conjugate
  gradient computations.
\newblock {\em SIAM J. Mat. Anal. Appl.}, 13(1):121--137, 1992.

\bibitem{halko2011finding}
N.~Halko, P.-G. Martinsson, and J.A. Tropp.
\newblock Finding structure with randomness: Probabilistic algorithms for
  constructing approximate matrix decompositions.
\newblock {\em SIAM Rev.}, 53(2):217--288, 2011.

\bibitem{hallman2020estimating}
E.~Hallman.
\newblock Estimating the backward error for the least-squares problem with
  multiple right-hand sides.
\newblock {\em Lin. Alg. Appl.}, 605:227--238, 2020.

\bibitem{Higham2002}
N.~J. Higham.
\newblock {\em {Accuracy and Stability of Numerical Algorithms}}.
\newblock SIAM, 2002.

\bibitem{higham2019new}
N.~J. Higham and T.~Mary.
\newblock A new approach to probabilistic rounding error analysis.
\newblock {\em SIAM J. Sci. Comput.}, 41(5):A2815--A2835, 2019.

\bibitem{higham2020sharper}
N.~J. Higham and T.~Mary.
\newblock Sharper probabilistic backward error analysis for basic linear
  algebra kernels with random data.
\newblock {\em SIAM J. Sci. Comput.}, 42(5):A3427--A3446, 2020.

\bibitem{ipsen2014effect}
I.~C.~F. Ipsen and T.~Wentworth.
\newblock The effect of coherence on sampling from matrices with orthonormal
  columns, and preconditioned least squares problems.
\newblock {\em SIAM J. Matrix Anal. Appl.}, 35(4):1490--1520, 2014.

\bibitem{martinsson2020randomized}
P.-G. Martinsson and J.~A. Tropp.
\newblock Randomized numerical linear algebra: Foundations and algorithms.
\newblock {\em Acta Numer.}, 29:403--572, 2020.

\bibitem{meier2022randomized}
M.~Meier and Y.~Nakatsukasa.
\newblock {Randomized algorithms for Tikhonov regularization in linear least
  squares}.
\newblock {\em arXiv:2203.07329}, 2022.

\bibitem{Meng2014}
X.~Meng, M.~A. Saunders, and M.~W. Mahoney.
\newblock {LSRN: A parallel iterative solver for strongly over- or
  underdetermined systems}.
\newblock {\em SIAM J. Sci. Comput.}, 36(2), 2014.

\bibitem{meurant2006lanczos}
G.~Meurant and Z.~Strako{\v{s}}.
\newblock The {L}anczos and conjugate gradient algorithms in finite precision
  arithmetic.
\newblock {\em Acta Numer.}, 15:471--542, 2006.

\bibitem{musco2018stability}
C.~Musco, C.~Musco, and A.~Sidford.
\newblock Stability of the lanczos method for matrix function approximation.
\newblock In {\em Proceedings of the Twenty-Ninth Annual ACM-SIAM Symposium on
  Discrete Algorithms}, pages 1605--1624. SIAM, 2018.

\bibitem{paige1982lsqr}
C.~C. Paige and M.~A. Saunders.
\newblock {LSQR: An algorithm for sparse linear equations and sparse least
  squares}.
\newblock {\em ACM Trans. Math. Soft.}, 8(1):43--71, 1982.

\bibitem{rokhlin2008fast}
V.~Rokhlin and M.~Tygert.
\newblock A fast randomized algorithm for overdetermined linear least-squares
  regression.
\newblock {\em Proc. Nat. Acad. Sci.}, 105(36):13212--13217, 2008.

\bibitem{stewartLS}
G.-W. Stewart.
\newblock Stability of the lanczos method for matrix function approximation.
\newblock In {\em Research, Development, and LINPACK, Mathematical Software
  III}, pages pp. 1--14. Academic Press, 1977.

\bibitem{tropp2011improved}
J.~A. Tropp.
\newblock {Improved analysis of the subsampled randomized Hadamard transform}.
\newblock {\em Adv. Adapt. Data Anal.}, 3(01n02):115--126, 2011.

\bibitem{tropp2022randomized}
J.~A. Tropp.
\newblock {Randomized block Krylov methods for approximating extreme
  eigenvalues}.
\newblock {\em Numer. Math.}, 150(1):217--255, 2022.

\bibitem{vershynin2010introduction}
R.~Vershynin.
\newblock Introduction to the non-asymptotic analysis of random matrices.
\newblock {\em arXiv:1011.3027}, 2010.

\bibitem{walden1995optimal}
B.~Wald{\'e}n, R.~Karlson, and J.-G. Sun.
\newblock Optimal backward perturbation bounds for the linear least squares
  problem.
\newblock {\em Numer. Lin. Alg. Appl.}, 2(3):271--286, 1995.

\bibitem{WathenReesetna08}
A.~J. Wathen and T.~Rees.
\newblock Chebyshev semi-iteration in preconditioning for problems including
  the mass matrix.
\newblock {\em Electron. Trans. Numer. Anal}, 34:125--135, 2008.

\bibitem{woodruff2014sketching}
D.~P. Woodruff.
\newblock Sketching as a tool for numerical linear algebra.
\newblock {\em Foundations and Trends{\textregistered} in Theoretical Computer
  Science}, 10(1--2):1--157, 2014.

\bibitem{Yamamoto2015}
Y.~Yamamoto, Y.~Nakatsukasa, Y.~Yanagisawa, and T.~Fukaya.
\newblock {Roundoff error analysis of the Cholesky QR2 algorithm}.
\newblock {\em Electron. Trans. Numer. Anal}, 44:306--326, 2015.

\end{thebibliography}
\end{document}